\documentclass[a4paper, 10pt, twoside, notitlepage]{amsart}

\usepackage{amsmath,amscd}
\usepackage{amssymb}
\usepackage{amsthm}
\usepackage{comment}
\usepackage{graphicx}
\usepackage{epstopdf} 
\usepackage{mathrsfs}
\usepackage{xcolor}
\usepackage[ocgcolorlinks, linkcolor=blue]{hyperref}

\usepackage[ocgcolorlinks,linkcolor=blue]{hyperref}

\newcommand{\LC}{\left(}
\newcommand{\RC}{\right)}

\theoremstyle{plain}
\newtheorem{thm}{Theorem}[section]
\newtheorem{prop}{Proposition}[section]

\newtheorem{lem}[prop]{Lemma}
\newtheorem{cor}[prop]{Corollary}

\newtheorem{rmk}[prop]{Remark}

\newtheorem{example}[prop]{Example}

\numberwithin{equation}{section}
\newcommand {\R} {\mathbb{R}} 
 \newcommand {\N} {\mathbb{N}}
 
\newcommand {\p} {\partial}
\newcommand{\vareps}{\varepsilon}
\newcommand{\eps}{\epsilon}

\newcommand{\norm}[1]{\lVert #1 \rVert}

\newcommand{\wt}{\widetilde}

\newcommand{\ccdot}{\,\cdot\,}
\newcommand{\s}{\hspace{0.5pt}}

\newcommand{\kommentar}[1]{}

\title[Uniqueness and gauge breaking for semilinear elliptic equations]{Uniqueness results and gauge breaking for inverse source problems of semilinear elliptic equations}

\author[T. Liimatainen]{Tony Liimatainen}
\address{Department of Mathematics and Statistics, University of Helsinki, Helsinki, Finland}
\curraddr{}
\email{tony.liimatainen@helsinki.fi}

\author[Y.-H. Lin]{Yi-Hsuan Lin}
\address{Department of Applied Mathematics, National Yang Ming Chiao Tung University, Hsinchu, Taiwan}
\curraddr{}
\email{yihsuanlin3@gmail.com}

\begin{document}
	\maketitle
	
	\begin{abstract}
	We study inverse source problems associated to semilinear elliptic equations of the form
	\[
	\Delta u(x)+a(x,u)=F(x),
	\]
	on a bounded domain $\Omega\subset \R^n$, $n\geq 2$.  We show that it is possible to use nonlinearity to break the gauge symmetry of the inverse source problem for a class of nonlinearities $a(x,u)$. This is in contrast to inverse source problems for linear equations, which always have a gauge symmetry. The class of  nonlinearities include certain polynomials and exponential nonlinearities. For these nonlinearities, we determine both $a(x,u)$ and $F(x)$ uniquely from the associated DN map.
	
    Moreover, for general nonlinearities $a(x,u)$, we show that we can recover the derivatives $\p_u^ka(x,u)$  and the source $F(x)$  up to a gauge. Especially, we recover general polynomial nonlinearities up to a gauge and generalize results of \cite{FO19,LLLS2019partial} by removing the assumption that $u\equiv 0$ is a solution.

		\medskip
		
		\noindent{\bf Keywords.} Inverse problems, inverse source problems, gauge invariance, semilinear elliptic equations, sine-Gordon equation, higher order linearization.
		
	 	\noindent{\bf Mathematics Subject Classification (2010)}: 35R30, 35J25, 35J61
		
	\end{abstract}

    \tableofcontents

	\section{Introduction}
	
	Let $\Omega \subset\R^n$ be a bounded domain with $C^\infty$-smooth boundary $\p \Omega$ with $n\geq 2$. In this paper we consider semilinear elliptic equations of the form
	\begin{align}\label{main equation}
		\begin{cases}
			\Delta u +a(x,u)=F &\text{ in }\Omega, \\
		u=f &\text{ on }\p \Omega,
		\end{cases}
	\end{align}
	where $a=a(x,z):\overline{\Omega}\times \R\to \R$ is $C^\infty$-smooth in  the $z$-variable. For presentational purposes we also assume
	 \begin{align}\label{a(x,0)=0}
		a(x,0)=0 \text{ in }\Omega.
	\end{align}  
	 This condition is not a restriction of generality as it can be achieved by redefining the source $F$ in \eqref{main equation}.  
    
Let us assume for now that the boundary value problem \eqref{main equation} is well-posed on an open subset $\mathcal{N}\subset C^{2,\alpha}(\p \Omega)$.
In this case, the \emph{Dirichlet-to-Neumann map} (DN map) is defined by the usual assignment
   \begin{equation}\label{eq:DNmap}
   \Lambda_{a,F}: \mathcal{N}\to C^{1,\alpha}(\p \Omega), \qquad f\mapsto \left. \p_\nu u_f\right|_{\p \Omega}.
   \end{equation}
     Here $\nu$ denotes the unit outer normal on $\p \Omega$. In Theorem \ref{Thm:wellposedness_and_expansion} we show that if there is 
     \[
f_0\in  C^{2,\alpha}(\p \Omega),
\]
such that the equation \eqref{main equation} admits a solution $u_{0}\in C^{2,\alpha}(\Omega)$ with $u_0|_{\p \Omega}=f_0$, and 
\begin{align*}
			0 \text{ is not a Dirichlet eigenvalue of }\Delta + \p _z a (x,u_0) \text{ in }\Omega,
		\end{align*}
then there is an open neighborhood $\mathcal{N}\subset C^{2,\alpha}(\p \Omega)$ of $f_0$ where \eqref{main equation} is well-posed in the following sense: For each $f\in \mathcal{N}$ there exists a solution $u_f$ to \eqref{main equation} with $u_f|_{\p \Omega}=f$ and the solution $u_f$ is unique in a fixed neighborhood of $u_0\in  C^{2,\alpha}(\Omega)$. If there holds the sign condition
\begin{align*}
 \p_za(x,z)\leq 0, 
\end{align*}
for $x\in \Omega$ and  $z\in \R$, the assumptions of Theorem \ref{Thm:wellposedness_and_expansion} will be satisfied and the DN map is well-defined by \cite{gilbarg2015elliptic}. If $F$ vanishes on $\Omega$, one can take $f_0\equiv 0$ on $\p \Omega$. In this case, Theorem \ref{Thm:wellposedness_and_expansion} reduces to similar well-posedness theorems in the literature, such as the one in \cite{LLLS2019partial}.

Consider the equation \eqref{main equation} for two sets $(a_1,F_1)$ and $(a_2, F_2)$ of coefficients. Let $\Lambda_1$ and $\Lambda_2$ be the corresponding DN maps defined on $\mathcal{N}_1\subset C^{2,\alpha}(\p \Omega)$ and $\mathcal{N}_2\subset C^{2,\alpha}(\p \Omega)$, respectively. When we write
\[
 \Lambda_1(f)=\Lambda_2(f), \text{ for any } f\in \mathcal{N}
\]
we have especially assumed that $\mathcal{N}\subset \mathcal{N}_1\cap \mathcal{N}_2$. 
   
   \smallskip
   
 \begin{itemize}
 	\item   \textbf{Inverse source problem:} What can we determine about both $a$ and $F$ from the knowledge of the corresponding DN map $\Lambda_{a,F}$?
 \end{itemize}
	
	\smallskip	

For general nonlinearities $a(x,z)$ it is impossible to determine both $a(x,z)$ and $F(x)$ simultaneously from the corresponding DN map. This is due to an inherit gauge invariance of the problem, which we will explain later. For inverse source problems of related linear equations, where the aim is to determine a source function from boundary measurements, the gauge invariance of the problem is well-known: 
	\begin{rmk}\label{rmk:counterexample}
	Let us consider the inverse source problem for the linear equation
	\begin{align}\label{l1}
		\begin{cases}
			\Delta u  +qu=F &\text{ in }\Omega,\\
			u=f& \text{ on }\p \Omega.
		\end{cases}
	\end{align}
	In this inverse problem one asks if the DN map $\Lambda_F: C^{\infty}(\p \Omega)\to C^{\infty}(\p \Omega)$ associated to the above equation determines $F$ uniquely. We assume here for simplicity that the potential function $q$ is known.  In general, the answer to the question is negative due to the following observation. Let $u$ solve \eqref{l1} and let $\psi$ be an arbitrary $C^2$-function satisfying $\left. \psi \right|_{\p \Omega} =\left. \p_\nu \psi \right|_{\p \Omega} =0$. Let us also define
     \begin{align}\label{eq_r1}
     	\tilde u:=u+\psi. 
     \end{align}
	Consequently, we have $\LC \tilde u|_{\p \Omega}, \p_\nu \tilde u|_{\p \Omega}\RC=\LC u|_{\p \Omega}, \p_\nu u|_{\p \Omega}\RC$, and 
	\begin{align}\label{eq_r2}
		\begin{split}
			\Delta \tilde u+ q \tilde u =&\Delta (u+\psi) + q(u+\psi)\\
			=&F -qu +\Delta \psi +qu+q\psi\\
			=&F+\Delta \psi +q\psi.
		\end{split}
	\end{align}
	Hence $u$ and $\tilde u$ solve the equations $\Delta u +qu=F$ and $\Delta \tilde u +q\tilde u=F+\Delta \psi+q\psi$ respectively. Since $u$ and $\tilde u$ also have the same Cauchy data on $\p \Omega$, it follows that the corresponding DN maps are the same: $\Lambda_{F}(f)=\Lambda_{F+\Delta \psi + q\psi}(f)$ on $\p \Omega$. It is thus not possible to determine a source function uniquely from the DN map.
\end{rmk}

In this work, we consider different types of nonlinearities, including general ones. For general nonlinearities $a(x,z)$, we prove in Theorem \ref{Thm: general nonlinearity} that the corresponding DN map determines the quantities
\begin{align}\label{same Taylor coef_intro}
		\p_z^k a(x,u_0(x)), \quad x\in \Omega, \quad k\in \N.
\end{align}
Here $u_0$ is a solution to \eqref{main equation} corresponding to a boundary value $f_0$. As already evidenced by Remark \ref{rmk:counterexample}, it might not be possible to recover $u_0$ from the DN map. This means that in general the condition \eqref{same Taylor coef_intro} does not determine $a(x,z)$, or even its derivatives in the variable $z$. 

Due to the above obstruction to determining $a(x,z)$, and consequently $F(x)$, in general, we mainly focus on nonlinearities $a(x,z)$ of the following special types:
\begin{itemize}
	\item \textbf{General polynomial nonlinearity:} 
	\begin{align}\label{poly in intro}
		a(x,z)=\displaystyle\sum_{k=1}^N a^{(k)}(x)z^k, \quad N\in \N,
	\end{align}
	\item \textbf{Exponential type nonlinearities:} 
	\begin{align}\label{expo in intro}
		a(x,z)=q(x)e^z \text{ and } a(x,z)=q(x)\s ze^z,
	\end{align} 
	\item \textbf{Sine-Gordon nonlinearity:} 
	\begin{align}\label{sine-Gordon in intro}
		a(x,z)=q(x)\sin(z).
	\end{align}
\end{itemize}
For these nonlinearities, we show that the corresponding inverse source problems are either uniquely solvable or there is a gauge symmetry, which has an explicit form. The fact that there are nonlinearities for which the related inverse source problem is uniquely solvable is in contrast to inverse source problem for linear equations, which always have the gauge symmetry presented in Remark \ref{rmk:counterexample}. That is, nonlinearity can make inverse source problems uniquely solvable.

Quadratic nonlinearity 
\[
 a(x,u)=a^{(1)}(x)u(x) + a^{(2)}(x)u^2(x)
\]
has a specific form gauge symmetry, which we now derive. 
For this, let us assume that $u$ solves \eqref{main equation}, where $a(x,z)$ is as above. Let $\psi\in C^{2}(\overline{\Omega})$. We denote by $\tilde a^{(1)}$, $\tilde a^{(2)}$ and $\tilde F$ another set of $C^\infty(\overline{\Omega})$ functions, which can depend on $\psi$. Let $\Lambda$ and $\tilde\Lambda$ be the DN maps corresponding to the coefficients without and with tilde signs respectively.  If we define 
\[
\tilde u:=u+\psi,
\]
then we have the chain of equivalences
\begin{align*}
	&\Delta  \tilde u + \tilde a^{(1)}\tilde u+  \tilde a^{(2)} \tilde u^2 = \tilde F, \\
	\iff & \Delta \LC u_1 +\psi \RC  + \tilde a^{(1)} \LC u+\psi \RC + \tilde a^{(2)} \LC u +\psi \RC^2  = \tilde F, \\
	\iff & \Delta u +\Delta \psi +\tilde a^{(1)} u + \tilde a^{(1)}\psi + \tilde a^{(2)} u^2 +2 \tilde a^{(2)}\psi u + \tilde a^{(2)}\psi^2 = \tilde F. 
\end{align*}
By using $\Delta u =-a^{(1)}u-a^{(2)}u^2 +F$ and equating the powers of $u$ gives the following system 
\begin{align}\label{gauge_intro}
	\begin{cases}
		F+\Delta \psi+ a^{(1)} \psi+a^{(2)} \psi^2= \tilde F & \text{ in }\Omega, \\
		a^{(1)}= \tilde a^{(1)}  + 2 \tilde a^{(2)} \psi & \text{ in }\Omega,\\
		a^{(2)}=\tilde a^{(2)}& \text{ in }\Omega.
	\end{cases}
\end{align}
If the above system is satisfied, then 
\[
 \Delta u + a(x,u)=F \iff \Delta \tilde u + a(x,\tilde u)=\tilde F.
\]
Consequently, if we additionally require that $\psi|_{\Omega}=\p_\nu \psi|_{\p \Omega}=0$, then the DN maps $\Lambda$ and $\tilde \Lambda$ are the same. That is, if we change the coefficients $\big( a^{(1)},a^{(2)},F \big)$ to $\big(  \tilde a^{(1)}, \tilde a^{(2)}, \tilde F\big)$, the DN map is preserved. Thus, it is at best possible to determine coefficients and a source from the DN map up to the gauge conditions \eqref{gauge_intro}. 


\smallskip

\noindent $\bullet$ \textbf{Earlier works.}
Before going into our results in detail, we discuss earlier related works. 
 The standard approach in the study of inverse problems for nonlinear elliptic equations was initiated in \cite{isakov1993uniqueness_parabolic}. There the author linearized the nonlinear DN map $C^{\infty}(\p \Omega)\to C^{\infty}(\p \Omega)$. The linearization reduced the inverse problem of a nonlinear equation to an inverse problem of a linear equation, which the author was able to solve by using methods for linear equations. 
Later, second order linearizations, where data depends on two independent parameters, were used to solve inverse problems for example in  \cite{AYT2017direct,CNV2019reconstruction,KN002,sun1996quasilinear,sun2010inverse,sun1997inverse}. 

For the case $F=0$ in $\Omega$ in \eqref{main equation}, equivalent to $u\equiv 0$ being a solution, inverse problems for semilinear elliptic equations were recently considered in \cite{FO19,LLLS2019partial}. The novelty of these works is that instead viewing nonlinearity as an additional complication in the inverse problem, the works used nonlinearity as a beneficial tool. The method of these two works originates from the seminal work \cite{KLU2018}, where inverse problems for nonlinear equations were studied in Lorentzian spacetimes. By using the method where nonlinearity is used as a tool, inverse problems for nonlinear equations have been solved in cases where the corresponding inverse problems for linear equations are still open. The method is by now usually called \emph{the higher order linearization method}.

After the works \cite{KLU2018, FO19,LLLS2019partial}, the literature about inverse problems for nonlinear equations based on the higher order linearization method, has grown substantially.
 The works \cite{LLLS2019partial,LLLS2019nonlinear,LLST2022inverse,KU2019remark,KU2019partial,FLL2021inverse,harrach2022simultaneous} investigated inverse problems for semilinear elliptic equations with general nonlinearities and in the case of partial data. Inverse problems for quasilinear elliptic equations using higher order linearization have been studied in \cite{KKU2022partial,CFKKU2021calderon,FKU2021inverse}. The works \cite{CLLO2022inverse,nurminen2022inverse} studied inverse problems for minimal surface equations on Riemannian surfaces and Euclidean domains. We also mention the works \cite{LL2020inverse,lin2020monotonicity,LL2022inverse,lai2019global,LO2022inverse,LZ2021inverse}, where inverse problems for semilinear fractional type equations have been studied.
 
 Inverse source problems for linear equations that regard determination of both unknown sources and coefficients have attracted recent interest. Applications of them include the photo/thermo-acoustic tomography \cite{A1}, magnetic anomaly detection \cite{A3,A4} and quantum mechanics \cite{A5,A6}. In this paper, we are interested in related nonlinear counterparts of the above works considering linear models. Finally, inverse problems of simultaneously recovering for both nonlinearities and initial data have been considered by \cite{LLLZ2021simultaneous} and \cite{LLL2021determining}.

\smallskip

In our first result we show that a quadratic nonlinearity and a source are determined by the corresponding DN map up to the gauge conditions \eqref{gauge_intro}.

\begin{thm}\label{Thm: gauge with quadratic}
	Let $\Omega \subset\R^n$ be a bounded domain with $C^\infty$-smooth boundary $\p \Omega$, $n\geq 2$. For $j=1,2$, let  
	\[
	a_j(x,z)=a_j^{(1)}(x)z+a_j^{(2)}(x)z^2,
	\]
	where $a_j^{(1)},a_j^{(2)}\in C^{\alpha}(\overline{\Omega})$ for some $0<\alpha <1$.
	 Consider the following semilinear elliptic equation 
\begin{align}\label{equation in thm}
	\begin{cases}
		\Delta u_j+ a_j(x,u_j) = F_j & \text{ in }\Omega,\\
		u_j=f &\text{ on }\p \Omega,
	\end{cases}
\end{align}
and let $\Lambda_{a_j,F_j}$ to be the corresponding DN map of \eqref{equation in thm}  for $j=1,2$. Suppose that there is an open set $\mathcal{N}\subset C^{2,\alpha}(\p\Omega)$ such that
\[
\Lambda_{a_1, F_1}(f)=\Lambda_{a_2, F_2}(f) \text{ for any } f\in \mathcal{N}. 
\]
Then there exists $\psi \in C^{2,\alpha}(\overline{\Omega})$ with $\psi|_{\p \Omega}=\left. \p_\nu \psi \right|_{\p \Omega}=0$ in $\Omega$ such that  
\begin{align}\label{gauge 1}
	\begin{cases}
		a_1^{(2)} =a_2^{(2)}=:a^{(2)}, &\\ 
		a_1^{(1)}=a_2^{(1)}+2a^{(2)} \psi, &\\ 
		F_1=F_2-\Delta\psi -a_1^{(2)}\psi -a^{(2)}\psi^2.& \\
	\end{cases}
\end{align}
\end{thm}
We remark that in the theorem above it is sufficient that the domain $\mathcal{N}$ of the DN maps is any non-empty open subset of $C^{2,\alpha}(\p\Omega)$. Especially $\mathcal{N}$ can by very small in size. The same holds for other results of this paper.

 Interestingly, if one a priori knows the linear term, the gauge symmetry of the inverse source problem can be broken. This results in unique determination:
\begin{cor}[Gauge breaking quadratic]\label{Cor: Unique determination}
	Assume as in Theorem \ref{Thm: gauge with quadratic} and adopt its notation. Assume additionally that 
\[
 a_1^{(1)}=a_2^{(1)} \text{ in } \Omega
\]
and 
\[
 a_1^{(2)}(x)\neq  0 \text{ or } a_2^{(2)}(x)\neq  0 \text{ at any } x\in \Omega.
\]
Then also
\[
 F_1=F_2 \text{ and } a_1^{(2)}=a_2^{(2)}   \text{ in } \Omega.
\]
\end{cor}
More precisely, the above corollary in particularly says the following. The inverse source problem of recovering $F$ from the DN map of
\[
 \Delta u  +qu+ u^2=F, 
\]
where $q$ is assumed to be known is uniquely solvable.  This is in contrast to the inverse source problem of $\Delta u  +qu=F$, which always has a gauge symmetry by Remark \ref{rmk:counterexample} (even if $q$ is known).  Thus, we have given an example where nonlinearity can be used to break the gauge symmetry in an inverse source problem.

For cubic nonlinearities we prove: 
 \begin{thm}\label{Thm: gauge with cubic}
 	Let $\Omega \subset\R^n$ be a bounded domain with $C^\infty$-smooth boundary $\p \Omega$,  $n\geq 2$. For $j=1,2$, let also 
 	$$
 	a_j(x,z)=a_j^{(1)}(x)z+a_j^{(2)}(x)z^2+a_j^{(3)}z^3,
 	$$
 	where $a_j^{(1)},a_j^{(2)},a_j^{(3)}\in C^{\alpha}(\overline{\Omega})$ for some $0<\alpha <1$. Let $\Lambda_{a_j,F_j}$ be the DN map of the equation
 	\begin{align}\label{equation in thm cubic}
 		\begin{cases}
 			\Delta u_j+ a_j(x,u_j) = F_j & \text{ in }\Omega,\\
 			u_j=f &\text{ on }\p \Omega.
 		\end{cases}
 	\end{align}
 	Suppose that there is an open set $\mathcal{N}\subset C^{2,\alpha}(\p\Omega)$ such that 
\[
\Lambda_{a_1, F_1}(f)=\Lambda_{a_2, F_2}(f) \text{ for any } f\in \mathcal{N}. 
\]
Then there exists $\psi \in C^{2,\alpha}(\overline{\Omega})$ with $\psi|_{\p \Omega}=\left. \p_\nu \psi \right|_{\p \Omega}=0$ in $\Omega$ such that

 	\begin{align}\label{gauge cubic thm}
 		\begin{cases}
 			a_1^{(3)}=a_2^{(3)}=:a^{(3)} &\text{ in }\Omega,\\
 			a_1^{(2)}=a_2^{(2)} +3a^{(3)}\psi&\text{ in }\Omega, \\
 			a_1^{(1)}= a_2^{(1)}  + 2 a_2^{(2)} \psi +3a^{(3)}\psi^2&\text{ in }\Omega,\\
 			F_1= F_2-\Delta \psi- a_2^{(1)} \psi-a_2^{(2)} \psi^2-a^{(3)}\psi^3  &\text{ in }\Omega.
 		\end{cases}
 	\end{align}
 \end{thm}

We also consider the case of general polynomial nonlinearities. In the following theorem we denote by 
	$$
	 \LC \begin{matrix}
		m\\
		k\\\end{matrix}  \RC=\dfrac{m!}{(m-k)! k!}
	$$ 
	the usual binomial coefficients. We include a converse statement to the result.

\begin{thm}\label{Thm: gauge polynomial}
	Let $\Omega \subset\R^n$ be a bounded domain with $C^\infty$-smooth boundary $\p \Omega$,  $n\geq 2$. For $j=1,2$, let $a_j(x,z)$ be a polynomial of the form
	\begin{align}\label{a_j poly}
		a_j(x,z)=\sum_{k=1}^N a_j^{(k)}(x)z^k \quad \text{ for }(x,z)\in \overline{\Omega}\times \R,
	\end{align}
    for some $N\in \N$, where $a_j^{(k)}\in C^{\alpha}(\overline{\Omega})$, for $j=1,2$ and $k=1,\ldots, N$. 
	 Given $F_j\in C^\alpha(\overline{\Omega})$ for some $0<\alpha <1$. Let $\Lambda_{a_j,F_j}$ be the DN map of the equation 
	\begin{align}\label{eq poly gauge inva}
		\begin{cases}
			\Delta u_j +a_j (x,u_j)=F_j &\text{ in }\Omega, \\
			u_j=f&\text{ on }\p \Omega.
		\end{cases}
	\end{align}
	Suppose that there is an open set $\mathcal{N}\subset C^{2,\alpha}(\p\Omega)$  such that 
\[
\Lambda_{a_1, F_1}(f)=\Lambda_{a_2, F_2}(f) \text{ for any } f\in \mathcal{N}. 
\]
Then there exists $\psi \in C^{2,\alpha}(\overline{\Omega})$ with $\psi|_{\p \Omega}=\left. \p_\nu \psi \right|_{\p \Omega}=0$ in $\Omega$ such that 
	\begin{align}\label{gauge poly}
		a_1^{(N-k)}=\sum_{m=N-k}^N  \LC \begin{matrix}
			m\\
			N-k\\\end{matrix}  \RC a_2^{(m)}\psi ^{m-N+k} \quad \text{ and }\quad  a_1^{(N)}=a_2^{(N)}\quad \text{ in }\Omega,
	\end{align}
	for $k=1,2,\ldots, N$, and 
	\begin{equation}\label{F_rel_poly}
	        F_1=F_2-\Delta\psi-\sum_{k=1}^N a_2^{(k)}\psi^k.                                           
    \end{equation}

Conversely, if \eqref{gauge poly} and \eqref{F_rel_poly} hold for some $\psi \in C^{2,\alpha}(\overline\Omega)$ with $\psi|_{\p \Omega}=\left. \p_\nu \psi \right|_{\p \Omega}=0$, then $\Lambda_{a_1,F_1}(f)=\Lambda_{a_2,F_2}(f)$ for all $f\in C^{2,\alpha}(\p \Omega)$ for which either side of the equation is defined.
\end{thm}
We remark that we could have also let $N$ to be finite, but otherwise unknown, in the assumptions of the above theorem. That is, $N$ could be initially assumed to be different for coefficients $(a_1,F_1)$ and $(a_2,F_2)$. The determination result, given by \eqref{gauge poly} and \eqref{F_rel_poly}, is the same also in this case. In addition, from Theorems \ref{Thm: gauge with quadratic} to \ref{Thm: gauge polynomial}, one can see that whether the gauge function $\psi$ exists or not, the coefficients of the highest order can be always determined uniquely.

Interestingly, also for cubic and general polynomial nonlinearities \eqref{a_j poly} the gauge invariances of the corresponding inverse source problems can sometimes be broken. In Corollary \ref{Cor: Unique determination} we gave the first example where nonlinearity can be used to break the gauge of an inverse source problem. The next result shows that inverse source problems for general polynomial nonlinearities are uniquely solvable if the second to highest order term is known.
\begin{cor}[Gauge breaking in the general case]\label{Cor: Unique determination poly}
	Assume as in Theorem \ref{Thm: gauge polynomial} and adopt its notation. Suppose additionally that 
\[
 a_1^{(N-1)}=a_2^{(N-1)} \text{ in } \Omega
\]
and 
\[
 \text{ either }\quad a_1^{(N)}(x)\neq  0 \quad  \text{ or } \quad  a_2^{(N)}(x)\neq  0 \quad  \text{ for all } x\in \Omega.
\]
Then all the  coefficients  are uniquely determined:
\[
 F_1\equiv F_2 \quad \text{ and } \quad a_1^{(k)}\equiv a_2^{(k)} \text{ in }\Omega, \quad k=1,2,\ldots,N.
\]
\end{cor}

In case it is a priori known that $F_1=F_2$, then we have:
\begin{cor}
	Let us adopt the notation and assumptions in Theorem \ref{Thm: gauge polynomial}. If $F_1=F_2$ in $\Omega$, then we have 
	\begin{align}\label{unique of coeff}
		a_1^{(k)}=a_2^{(k)} \text{ in }\Omega,
	\end{align}
for $k=1,2,\ldots ,N$.
\end{cor}
The above corollary in particularly says the following. If we consider inverse problem for the equation
\[
 \Delta u  +qu+ u^2=F, \quad F \text{ known},
\]
then we can recover the lower order term $q$ from the DN map.

We also study an inverse source problem for general semilinear elliptic equations and do not assume that the nonlinearity is necessarily a polynomial. In fact, we will prove Theorem \ref{Thm: general nonlinearity} below before Theorem \ref{Thm: gauge polynomial} for convenience.
\begin{thm}\label{Thm: general nonlinearity}
	Let $\Omega \subset\R^n$ be a bounded domain with $C^\infty$-smooth boundary $\p \Omega$,  $n\geq 2$.
	For $j=1,2$, let $a_j(\ccdot,z)\in C^\alpha(\overline{\Omega})$ satisfy the condition \eqref{a(x,0)=0}
	and assume that $a_j(x,z)$ is $C^\infty$-smooth with respect to the $z$-variable. Given $F_j\in C^\alpha(\overline{\Omega})$ for some $0<\alpha <1$, let $\Lambda_{a_j,F_j}$ be the DN map of 
	\begin{align}\label{semilinear ellip index}
		\begin{cases}
			\Delta u_j +a_j (x,u_j)=F_j &\text{ in }\Omega, \\
			u_j=f&\text{ on }\p \Omega.
		\end{cases}
	\end{align}
	Suppose that there is an open set $\mathcal{N}\subset C^{2,\alpha}(\p\Omega)$  such that 
\[
\Lambda_{a_1, F_1}(f)=\Lambda_{a_2, F_2}(f) \text{ for any } f\in \mathcal{N}. 
\]
	Then, for any $f_0\in \mathcal{N}$, we have 
	\begin{align}\label{same Taylor coef}
		\p_z^k a_1(x,u_1^{(0)}(x))= \p_z^k a_2(x,u_2^{(0)}(x)), \quad  x\in \Omega,
	\end{align}
	for any $k\in \N$. Here $u_1^{(0)}$ and $u_2^{(0)}$ are the solutions of \eqref{semilinear ellip index} with boundary condition $u_j^{(0)} \big|_{\p \Omega}=f_0$. 
\end{thm}

As a corollary to Theorem \ref{Thm: general nonlinearity}, we do case studies of inverse source problems when the nonlinearity of the model is either of exponential type or $a(x,z)=q(x)\sin(z)$. 
These models arise in mathematical modeling of combustion, where the nonlinearity involved is of exponential type (see e.g. \cite{volpert2014elliptic}). The nonlinearity $a(x,z)=q(x)\sin(z)$ corresponds to the sine-Gordon equation. The DN map and inverse problems for the sine-Gordon equation have been considered for example in \cite{BK1989inverse,FP2012dirichlet}. The models are chosen so to give  examples of cases where the inverse source problem is uniquely solvable, or has an explicit gauge symmetry.

Let $q$ and $F$ belong to $C^\alpha (\overline{\Omega})$, and consider the semilinear elliptic equations 
\begin{align}\label{equ exponential nonlinearity}
	\begin{cases}
		\Delta u + q(x)e^u =F &\text{ in }\Omega, \\
		u=f &\text{ on }\p \Omega.
	\end{cases}
\end{align}
and 
\begin{align}\label{equ exponential nonlinearity with u}
	\begin{cases}
		\Delta u + q(x)ue^u =F &\text{ in }\Omega, \\
		u=f &\text{ on }\p \Omega.
	\end{cases}
\end{align}
For the corresponding inverse source problems we assume that both of the above boundary value problems have a solution $u_0$ for some boundary value $f_0$ such that $0$ is not an eigenvalue of $\Delta + \p_za(x,u_0)$. 
%
In this case, it follows from Theorem \ref{Thm:wellposedness_and_expansion} that  the DN maps $\mathcal{N} \to C^{1,\alpha}(\p \Omega)$ are defined on an open subset $\mathcal{N}\subset C^{2,\alpha}(\p M)$ as before by
\[
u\mapsto \left. \p_\nu u_f\right|_{\p \Omega}.
\]
Here, $u_f$ is the unique solution on a neighborhood of $u_0$  to either \eqref{equ exponential nonlinearity} or \eqref{equ exponential nonlinearity with u} depending on which of the two models we are considering. 
We remark that in the case $q\leq 0$ the above assumptions are satisfied and Theorem \ref{Thm:wellposedness_and_expansion} holds  for \eqref{equ exponential nonlinearity} by \cite[Theorem 15.12]{gilbarg2015elliptic}. In this case, also \eqref{equ exponential nonlinearity with u}  has a solution $u_0$ for a given boundary value $f_0\in C^{2,\alpha}(\p \Omega)$ by \cite[Theorem 15.12]{gilbarg2015elliptic}. However, to apply Theorem \ref{Thm:wellposedness_and_expansion}, one still needs to assume that $0$ is not an eigenvalue of $\Delta + \p_za(x,u_0)$. We also remark that if $F$ is assumed to be small enough, the DN maps of \eqref{equ exponential nonlinearity} and \eqref{equ exponential nonlinearity with u} are well-defined by Proposition \ref{Prop: well-posed samll inputs} on a neighborhood of the zero boundary value.

For the nonlinearity $a(x,z)=q(x) e^z$, the inverse source problem is not uniquely solvable due to a gauge symmetry. However, if the nonlinearity is $q(x) ze^z$, and $q(x)\neq 0$ for $x\in \Omega$, the corresponding inverse source problem has a unique solution.

\begin{cor}\label{Cor: Exponential} Let $\Omega \subset\R^n$ be a bounded domain with $C^\infty$-smooth boundary $\p \Omega$,  $n\geq 2$.
Let $q_j\in C^{\alpha}(\overline{\Omega})$, and suppose  additionally that 

\smallskip

	\textbf{Case 1.}
	
	\[
	a_j(x,z)=q_j(x)e^z;
	\]
	
		\textbf{Case 2.}
		
		\[
		a_j(x,z)=q_j(x) ze^z,
		\]
	with $q_j\neq 0$ in $\Omega$, for $j=1,2$.
	
	\smallskip
	
	Suppose that there is an open $\mathcal{N}\subset C^{2,\alpha}(\p\Omega)$  such that the corresponding DN maps $\Lambda_{a_j,F_j}$ of the equation
\begin{align*}
		\begin{cases}
			\Delta u_j +a_j (x,u_j)=F_j &\text{ in }\Omega, \\
			u_j=f&\text{ on }\p \Omega
		\end{cases}
	\end{align*}
	satisfy
	\[
\Lambda_{a_1, F_1}(f)=\Lambda_{a_2, F_2}(f) \text{ for any } f\in \mathcal{N}. 
\] 
Then we have:

\smallskip 

\textbf{Case 1.} Gauge symmetry:
	\begin{align}\label{gauge exponential}
		q_1=q_2e^{\psi} \quad \text{ and }\quad F_1=F_2-\Delta \psi \text{ in }\Omega.
	\end{align}
Conversely, if  \eqref{gauge exponential} holds for some $\psi \in C^{2,\alpha}(\overline\Omega)$ with $\psi|_{\p \Omega}=\left. \p_\nu \psi \right|_{\p \Omega}=0$, then $\Lambda_{a_1,F_1}(f)=\Lambda_{a_2,F_2}(f)$ for all $f\in C^{2,\alpha}(\p \Omega)$ for which either side of the equation is defined. 

   \textbf{Case 2.} Unique determination:
    \begin{align}\label{uniqueness exponential}
   	q_1=q_2 \quad \text{ and }\quad F_1=F_2 \text{ in }\Omega.
   \end{align}

%
%
%
\end{cor}

As the second application of Theorem \ref{Thm: general nonlinearity}, we consider the inverse source problem for the elliptic sine-Gordon equation. Again, let $q$ and $F$ belong to $C^{\alpha}(\overline{\Omega})$, and assume that the equation 
\begin{align}\label{sine-Gordon equ}
	\begin{cases}
		\Delta u + q \sin u =F &\text{ in }\Omega, \\
		u=f &\text{ on }\p \Omega,
	\end{cases}
\end{align}
has a solution for some boundary value $f_0\in C^{2,\alpha}(\p \Omega)$ such that $0$ is not an eigenvalue of $\Delta + \p_za(x,u_0)$.  Then the equation is well-posed on a neighborhood $\mathcal{N}\subset C^{2,\alpha}(\p \Omega)$ of $f_0$ by Theorem \ref{Thm:wellposedness_and_expansion}. The DN map of \eqref{sine-Gordon equ} is again defined by
\[
\Lambda_{q,F} :\mathcal{N} \to C^{1,\alpha}(\p \Omega), \qquad u\mapsto \left. \p_\nu u_f\right|_{\p \Omega},
\]
where $u_f\in C^{2,\alpha}(\overline{\Omega})$ is the unique solution to \eqref{sine-Gordon equ} on a neighborhood of $u_{f_0}$. If $F$ is assumed to be small enough, the DN map of \eqref{equ exponential nonlinearity with u} is well defined by Proposition \ref{Prop: well-posed samll inputs} on a neighborhood of the zero boundary value.

For the sine-Gordon equation, the inverse source problem is solvable. 
\begin{cor}\label{Cor:sine-Gordon}
Let $\Omega \subset\R^n$ be a bounded domain with $C^\infty$-smooth boundary $\p \Omega$,  $n\geq 2$.
Let $q\in C^{\alpha}(\overline{\Omega})$, and suppose  additionally that 
\begin{align}
	a_j(x,z)=q_j(x)\sin z,
\end{align}
for $j=1,2$. Suppose that there is an open set $\mathcal{N}\subset C^{2,\alpha}(\p\Omega)$  such that the corresponding DN maps $\Lambda_{q_j, F_j}$ of the equation
%
\begin{align}\label{sine-Gordon thm}
	\begin{cases}
		\Delta u_j +q_j\sin (u_j)=F_j &\text{ in }\Omega, \\
		u_j=f&\text{ on }\p \Omega.
	\end{cases}
\end{align}
satisfy 
\[
\Lambda_{q_1, F_1}(f)=\Lambda_{q_2, F_2}(f) \text{ for any } f\in \mathcal{N}.
\]
Then 
\begin{align}\label{unique sine-Gordon cor}
	q_1=q_2 \quad \text{ and }\quad F_1=F_2 \text{ in }\Omega.
\end{align}
\end{cor}

The paper is organized as follows. In Section \ref{Sec 2}, we prove a well-posedness result for semilinear elliptic equations with sources. Moreover, a local well-posedness result is also given in Section \ref{Sec 2}, and the proof is left in Appendix \ref{Appendix}.
In Section \ref{Sec 3}, we prove Theorems \ref{Thm: gauge with quadratic}, \ref{Thm: gauge with cubic}, \ref{Thm: gauge polynomial} and \ref{Thm: general nonlinearity} by using the higher order linearization method. We prove Corollaries \ref{Cor: Exponential} and \ref{Cor:sine-Gordon} in Section \ref{Sec 4}. 

\section{Preliminaries}\label{Sec 2}
In this section, we prove a local well-posedness result for the Dirichlet problem \eqref{main equation} on a neighborhood of a given solution. Let $0<\alpha<1$ and $\delta >0$ and denote
\begin{align}\label{N_delta}
	\mathcal{N}_\delta := \left\{ f\in C^{2,\alpha}(\p \Omega): \, \norm{f}_{C^{2,\alpha}(\p \Omega)}\leq \delta  \right\}.
\end{align}
Note that when the source function $F$ of the equation $\Delta u(x) + a(x,u)=F(x)$ does not vanish, zero function is not a solution to the equation \eqref{main equation}. This is the main reason why our well-posedness result differs from the usual ones, such as the one in \cite{LLLS2019nonlinear,KU2019partial}.

\begin{thm}[Well-posedness]\label{Thm:wellposedness_and_expansion}
		Let $\Omega\subset \R^n$ be a bounded domain with $C^\infty$ boundary $\p \Omega$ and $n\geq 2$. Given $\alpha \in (0,1)$, $F \in C^{2,\alpha}(\overline{\Omega})$ and $f_0\in C^{2,\alpha}(\p \Omega)$, suppose that there exists a solution $u_0\in C^{2,\alpha}(\overline{\Omega})$ to   
		\begin{align}\label{solvability 0}
			\begin{cases}
				\Delta u_0+a(x,u_0)=F & \text{ in } \Omega,\\
				u= f_0 & \text{ on } \p \Omega.
			\end{cases}
		\end{align}		
		Assume also  that 
		\begin{align}\label{eigenvalue}
			0 \text{ is not a Dirichlet eigenvalue of }\Delta + \p _z a (x,u_0) \text{ in }\Omega.
		\end{align}
		Then there are $\delta>0$ and $C >0$ such that for any $f\in \mathcal{N}_\delta$ there exists a unique solution $u\in C^{2,\alpha}(\overline{\Omega})$ of  
		\begin{align}\label{solvability}
			\begin{cases}
				\Delta u+a(x,u)=F & \text{ in } \Omega,\\
				u= f_0+f & \text{ on } \p \Omega,
			\end{cases}
		\end{align}
		within the class $\left\{ w \in C^{2,\alpha}(\overline{\Omega}) :\, \norm{w-u_0}_{C^{2,\alpha}(\overline{\Omega})} \leq C \right\}$. 
		Moreover, there are $C^{\infty}$ Fr\'echet differentiable maps 
		\begin{align*}
			\begin{array}{rll}
				\mathcal{S}:& \!\!\!\mathcal{N}_{\delta} \to C^{2,\alpha}(\overline{\Omega}), \ & f \mapsto u, \\
				\Lambda:& \!\!\!\mathcal{N}_{\delta} \to C^{1,\alpha}(\p \Omega), \ &f \mapsto \left. \p_{\nu} u \right|_{\p \Omega}.
			\end{array}
		\end{align*}
	\end{thm}


	\begin{proof}
	We use the standard method, which uses the implicit function theorem in Banach spaces to prove the theorem. A similar proof can be found from the work \cite{LLLS2019nonlinear} where the source $F$ is assumed to vanish. We refer to that work for additional details of the arguments used.  
	Let 
	\[
	\mathcal{B}_1=C^{2,\alpha}(\p \Omega), \quad \mathcal{B}_2=C^{2,\alpha}(\overline{\Omega}), \quad \mathcal{B}_3=C^{\alpha}(\overline{\Omega})\times C^{2,\alpha}(\p \Omega)
	\]
	and assume that $u_0$ solves \eqref{solvability 0}. 
	Consider the map 
    \begin{align*}
    		\Psi:\mathcal{B}_1\times \mathcal{B}_2 &\to \mathcal{B}_3, \\
    	       (f,u) &\mapsto \LC \Delta u +a(x,u)-F, u|_{\p \Omega}-\LC f_0+f\RC  \RC.
    \end{align*}
	Similar to \cite[Section 2]{LLLS2019nonlinear}, one can show that the map $u \mapsto a(x,u)$ is a $C^{\infty}$ map from $C^{2,\alpha}(\overline{\Omega}) \to C^{2,\alpha}(\overline{\Omega})$. 
	
	Notice that $ \Psi(0,u_0)  =(0,0)$, where $u_0\in C^{2,\alpha}(\overline{\Omega})$ is a solution to \eqref{solvability 0}. The first linearization of $\Psi=\Psi(f,u)$ at $(0,u_0)$ in the variable $u$ is 
	\[
	\left. D_u \Psi \right|_{(0,u_0)}(v) = \LC \Delta v + \p_z a(x,u_0) v,v|_{\p \Omega}\RC.
	\]
	This is a homeomorphism $ \mathcal{B}_2\to  \mathcal{B}_3$ by the condition \eqref{eigenvalue}, which is guaranteed by well-posedness and Schauder estimates for linear second order elliptic equations.

	Using the implicit function theorem in Banach spaces~\cite[Theorem 10.6 and Remark 10.5]{renardy2006introduction} yields that there is $\delta>0$ and an open ball $\mathcal{N}_\delta\subset C^{2,\alpha}(\p \Omega)$ and a $C^{\infty}$ map $\mathcal{S}: \mathcal{N}_{\delta} \to \mathcal{B}_2$ such that whenever $\norm{f}_{C^{2,\alpha}(\p \Omega)}\leq \delta$ we have 
	\[
	\Psi(f,\mathcal{S}(f))=(0,0).
	\]
	Since $\mathcal{S}$ is smooth and $\mathcal{S}(0) = u_0$, the solution $u =\mathcal{S}(f)$ satisfies 
	\[
	\norm{u}_{C^{2,\alpha}(\overline{\Omega})}\leq C \norm{f}_{C^{2,\alpha}(\p \Omega)}.
	\]
	Moreover, by the uniqueness statement of the implicit function theorem, by redefining $\delta>0$ to be smaller if necessary,  $u=\mathcal{S}(f)$ is the only solution to $\Psi(f,u)=(0,0)$ whenever $\norm{f}_{C^{2,\alpha}(\p \Omega)}\leq \delta$ and 
	$$
	\norm{u}_{C^{2,\alpha}(\overline{\Omega})}\leq C.
	$$
	As in \cite{LLLS2019nonlinear}, one can check that the solution operator $\mathcal{S}: \mathcal{N}_{\delta} \to C^{2,\alpha}(\overline{\Omega})$ is a $C^{\infty}$ map in the Fr\'echet sense. Since the normal derivative is a linear map $C^{2,\alpha}(\overline{\Omega}) \to C^{1,\alpha}(\p \Omega)$, then  $\Lambda$ is also a well defined $C^{\infty}$ map $\mathcal{N}_{\delta} \to C^{1,\alpha}(\p \Omega)$.
	\end{proof}

We remark that if $a(x,z)$ satisfies
\[
 \p_za(x,z)\leq 0
\]
for all $x\in \Omega$ and $z\in \R$,  
then the conditions of Theorem \ref{Thm:wellposedness_and_expansion} are satisfied by \cite[Theorem 15.12]{gilbarg2015elliptic}. Next we give an example of well-posedness in the two-dimensional case, where monotonicity method work well due to Sobolev embedding.

\begin{example}
	In the two-dimensional case, let  $\Omega$ be a bounded domain with $C^\infty$-smooth boundary $\p \Omega$. We consider the semilinear equation 
	\begin{align}\label{cubic nonlinear equ}
		\begin{cases}
			-\Delta u + a^{(3)}u^3 =F  &\text{ in }\Omega, \\
			u=0 &\text{ on }\p \Omega,
		\end{cases}
	\end{align}
	where $c_0 \leq a^{(3)}\in C^\infty(\overline{\Omega})$, for some constant $c_0>0$. Given $F\in H^{-1}(\Omega)$, there exists a unique solution $u_F\in H^1(\Omega)$ solving \eqref{cubic nonlinear equ}. 

	The proof is by the monotone operator method, which works well in dimension two. Let us multiply \eqref{cubic nonlinear equ} by a test function $\varphi \in H^1_0(\Omega)$. Then an integration by parts yields 
	\begin{align*}
		\int_{\Omega} \nabla u \cdot \nabla \varphi \, dx + \int_{\Omega} a^{(3)}u^3 \varphi \, dx =\int_{\Omega} F\varphi \, dx.
	\end{align*}
	Let $\mathbf{T}:H^1_0(\Omega)\to H^{-1}(\Omega)$ be the operator given by 
	\begin{align*}
		\langle  \mathbf{T}u, \varphi   \rangle =	\int_{\Omega} \nabla u \cdot \nabla \varphi \, dx + \int_{\Omega} a^{(3)} u^3 \varphi \, dx , \text{ for any }\varphi \in H^1_0(\Omega).
	\end{align*}
	It is not hard to see that $\mathbf{T}u-F$ is the Frech\'et derivative of the energy functional 
	\begin{align*}
		\mathbf{E}(u)=\frac{1}{2}\int_{\Omega} |\nabla u|^2 \, dx +\frac{1}{4}\int_\Omega a^{(3)} u^4 \, dx-\int_{\Omega}Fu \, dx.
	\end{align*}
	 Since $\Omega\subset \R^2$, the Sobolev space $H^1(\Omega)$ embeds in $L^6(\Omega)$. Using this fact, one can show that the operator $\mathbf{T}$ is bounded, strictly monotone and coercive. Then by applying the classical energy method, the functional $\mathbf{E}$ is coercive and weakly lower semicontinuous on $H^1_0(\Omega)$ (for example, see \cite[Theorem 26.11]{fucik2014nonlinear}). Therefore, $\mathbf{E}$ is bounded from below and attains its infimum at some function $u\in H^1_0(\Omega)$, so that $u$ is a solution of \eqref{cubic nonlinear equ}. The uniqueness of $u$ is a direct result of the strict monotonicity of $\mathbf{T}$. We refer to  \cite[Theorem 3.1]{beretta2022nonlinear} for more details about this argument.
\end{example}

Under the assumptions of Theorem \ref{Thm:wellposedness_and_expansion} above, the boundary value problem \eqref{main equation} is well-posed in the following sense: There is $f_0\in C^{2,\alpha}(\p \Omega)$ and  $\delta>0$ such that for each $f\in f_0+\mathcal{N}_\delta$ there exists a solution $u_f$ to \eqref{main equation} with $u_f|_{\p\Omega}=f$. The solution $u_f$ is unique on a fixed neighborhood of $u_{0}\in C^{2,\alpha}(\Omega)$, where $u_0$ solves \eqref{main equation} with boundary value $f_0$. In this case the corresponding DN map $f_0+\mathcal{N}_\delta\to C^{1,\alpha}(\p \Omega)$ defined by the assignment $f\mapsto \left. \p_\nu u_f \right|_{\p \Omega}$ is well-defined and $C^\infty$ smooth in the Fr\'echet sense.

Before ending this section, we give a well-posedness result in the case when the Dirichlet data and the source $F$ are both  sufficiently small. We record it to provide an example where the DN map is always defined. Let 
\[
\mathcal{A}_\vareps:=\left\{ F\in C^{2,\alpha}(\overline{\Omega}): \, \norm{F}_{C^\alpha(\overline{\Omega})}\leq \varepsilon \right\}.
\]
We have following result, whose the proof is placed in the Appendix \ref{Appendix}.

\begin{prop}\label{Prop: well-posed samll inputs}
	Let $\Omega\subset \R^n$ be a bounded domain with $C^\infty$ boundary $\p \Omega$ and $n\geq 2$. Assume that $a(x,0)=0$. There are $C>0$, $\varepsilon>0$ and $\delta>0$ such that for any $F\in \mathcal{A}_\eps$ and $f\in \mathcal{N}_\delta$, then there is a unique  solution $u\in C^{2,\alpha}(\overline{\Omega})$ of 
	\begin{align}\label{solvability in prop}
		\begin{cases}
			\Delta u+a(x,u)=F & \text{ in } \Omega,\\
			u= f & \text{ on } \p \Omega,
		\end{cases}
	\end{align}
	within the class $\left\{ w \in C^{2,\alpha}(\overline{\Omega}) :\, \norm{w}_{C^{2,\alpha}(\overline{\Omega})} \leq C(\vareps+\delta) \right\}$.	Moreover, there is a $C^{\infty}$ Fr\'echet differentiable map 
	\begin{align*}
				\mathcal{S}: \, \mathcal{A}_\vareps\times \mathcal{N}_{\delta}\to C^{2,\alpha}(\overline{\Omega}), \quad   (F,f) \mapsto u.
		\end{align*}
	In particular, for a fixed $F\in \mathcal{A}_\vareps$, the map 
	\begin{align*}
			\Lambda_F: \, \mathcal{N}_{\delta} &\to C^{1,\alpha}(\p \Omega), \quad  f \mapsto \left. \p_{\nu} u \right|_{\p \Omega}
	\end{align*}
    is also $C^{\infty}$ Fr\'echet differentiable.
\end{prop}

\section{Uniqueness for polynomial nonlinearities up to gauge invariances}\label{Sec 3}	
\subsection{Quadratic nonlinearity}
In the introduction we showed that the inverse source problem for 
\[
 \Delta u + a(x,u)=F,
\]
where $a(x,u)$ is quadratic,
$$
a(x,u(x))=a^{(1)}(x)u(x) +a^{(2)}(x)u^2(x),
$$ 
has a gauge invariance given by the gauge conditions \eqref{gauge_intro}. We show next that these gauge conditions are the only obstruction to uniqueness in the inverse source problem for quadratic nonlinearities. This is Theorem \ref{Thm: gauge with quadratic}.

\medskip

For the quadratic nonlinearity we consider Dirichlet data of the form
\begin{align}\label{Dirichlet data for HOL}
	f:=f(x;\eps_1,\eps_2): =f_0(x)+\eps_1 f_1(x) +\eps_2 f_2(x)  \quad  x\in  \p \Omega,
\end{align}
where where $f_0,f_1, f_2 \in C^{2,\alpha}(\p \Omega)$, and $\eps_1$ and $\eps_2$ are small real parameters. 

\begin{proof}[Proof of Theorem \ref{Thm: gauge with quadratic}]
	By assumption there is $\mathcal{N}\subset C^{2,\alpha}(\p \Omega)$ such that
	\[
	 \Lambda_{a_1,F_1}(f)=\Lambda_{a_2,F_2}(f), \quad f\in \mathcal{N}.
	\]
    Let $f_0\in \mathcal{N}$, $f_1, f_2 \in C^{2,\alpha}(\p \Omega)$ and $\eps_1,\eps_2>0$ such that $f_0+\eps_1f_1+\eps_2f_2\in \mathcal{N}$. We apply the higher order linearization method to the equation
	\begin{align}\label{equ: zero order linear}
		\begin{cases}
			\Delta u_j + a_j^{(1)} u_j +a_j^{(2)} u_j^2 =F_j & \text{ in }\Omega ,\\
			u_j=f_0+\eps_1 f_1 +\eps_2 f_2  & \text{ on }\p \Omega.
		\end{cases}
	\end{align}
	We denote $\eps=(\eps_1, \eps_2)$, which especially means that $\eps=0$ is equivalent to $\eps_1 =\eps_2=0$. Below the index $j=1,2$ corresponds to the different sets of coefficients, and an index $\ell=1,2$ to $\eps_\ell$ parameters. Let us denote by $u_j^{(0)}$ the solution to 
	\begin{align}\label{equ: zero order linear zero}
		\begin{cases}
			\Delta u_j^{(0)} + a_j^{(1)} u_j^{(0)} + a_j^{(2)}\LC u_j^{(0)}\RC^2 =F & \text{ in }\Omega ,\\
			u_j^{(0)}=f_0  & \text{ on }\p \Omega.
		\end{cases}
	\end{align}

	With the well-posedness holding on a neighborhood $\mathcal{N}$ of $f_0$, see Theorem \ref{Thm:wellposedness_and_expansion}, we can differentiate \eqref{equ: zero order linear} with respect to $\eps_\ell$, for $\ell=1,2$. We obtain
	\begin{align}\label{equ: first order linear zero}
		\begin{cases}
			\LC \Delta  + a_j^{(1)}+2a_j^{(2)} u_j^{(0)} \RC v_j^{(\ell)} =0 & \text{ in }\Omega,\\
			v_j^{(\ell)} =f_\ell & \text{ on }\p \Omega,
		\end{cases}
	\end{align}
	where 
	\[
	v_j^{(\ell)}=\left. \p_{\eps_\ell} \right|_{\eps=0}u_j,
	\]
	for $j,\ell=1,2$.  It also follows from Theorem \ref{Thm:wellposedness_and_expansion} that we know the DN maps of the equation \eqref{equ: first order linear zero} for $j=1$ and $j=2$ agree. Thus, by the global uniqueness result for linear inverse boundary value problems (see e.g. \cite[Proposition 2.1]{LLLS2019partial} or \cite{sylvester1987global} for $n \geq 3$ and  \cite{bukhgeim2008recovering, blasten2017singular} for $n=2$),
	we have 
	\begin{align}\label{uniqueness of 1st coef}
	Q:=a_1^{(1)} +2a_1^{(2)} u_1^{(0)} = a_2^{(1)} +2a_2^{(2)} u_2^{(0)} \text{ in }\Omega.
	\end{align}
	It then follows by uniqueness of solutions to the Dirichlet problem \eqref{equ: zero order linear zero} that 
	\[
	v^{(\ell)}:=v_1^{(\ell)}=v_2^{(\ell)} \text{ in }\Omega,
	\]
	for $\ell=1,2$.

	We next derive the equation for the second order linearization of \eqref{equ: zero order linear} at $u_j^{(0)}$. For $j=1,2$, a straightforward computation shows that 
	\begin{align}\label{equ: second order linear zero}
		\begin{cases}
			\LC \Delta +a_j^{(1)}+2a_j^{(2)}u_j^{(0)} \RC w_j + 2a_j^{(2)}v^{(1)}v^{(2)}=0 & \text{ in }\Omega,\\
			w_j=0 &\text{ on }\p \Omega,
		\end{cases}
	\end{align}
	where 
	$$
	w_j=\left.\p^2 _{\eps_1 \eps_2} \right|_{\eps=0}u_j.
	$$
	We show next that $a_1^{(2)}=a_2^{(2)}$ in $\Omega$. 
	For that, let us consider $\mathbf{v}^{(\ell)}$ to be the solution of 
	\begin{align}\label{bold v in quadratic}
		\begin{cases}
			\LC \Delta +Q \RC \mathbf{v}^{(\ell)}=0 &\text{ in }\Omega,\\
			\mathbf{v}^{(\ell)}= g_\ell &\text{ on }\p \Omega,
		\end{cases}
	\end{align}
	where $Q$ is given in \eqref{uniqueness of 1st coef} and $g_\ell\in H^{1/2}(\p \Omega)$ will be chosen later for $\ell=1,2$.  We multiply \eqref{bold v in quadratic} by $\mathbf{v}^{(1)}$. Moreover, by using  $\p_{\nu}w_1=\p_\nu w_2$ on $\p \Omega$, integration by parts yields
	\begin{align*}
	0=&	\int_{\p \Omega} \LC\p_{\nu}w_1-\p_\nu w_2 \RC \mathbf{v}^{(1)}\, dS \\
	    =& \int_{\Omega} \Delta \LC w_1-w_2\RC \mathbf{v}^{(1)}\, dx + \int_{\Omega}\nabla \LC w_1-w_2\RC \cdot \nabla \mathbf{v}^{(1)}\, dx\\
	    =& \int_{\Omega} \Delta \LC w_1-w_2\RC \mathbf{v}^{(1)}\, dx + \int_{\p \Omega} \LC w_1-w_2\RC \cdot\p_\nu \mathbf{v}^{(1)}\, dS\\
	    &-\int_{\Omega} \LC w_1-w_2 \RC \Delta \mathbf{v}^{(1)}\, dx \\
		=&\int_{\Omega} \LC a_1^{(2)}-a_2^{(2)} \RC v^{(1)}v^{(2)}\mathbf{v}^{(1)} \, dx.
	\end{align*}
Here we used $w_1-w_2=0$ on $\p \Omega$ and \eqref{equ: second order linear zero} and  \eqref{bold v in quadratic}.
By using that products of pairs of solutions (CGOs) to \eqref{equ: first order linear zero} are dense in $L^1(\Omega)$ for $n\geq 2$,
we can choose $v^{(1)}$ and $v^{(2)}$ so that we obtain
\begin{equation}\label{eq:for_Runge}
\LC a_1^{(2)}-a_2^{(2)} \RC  \mathbf{v}^{(1)}=0 \text{ in } \Omega.
\end{equation}

Next we take also $\mathbf{v}^{(1)}$ as a CGO solution and multiply the above identity by yet another CGO solution $\mathbf{v}^{(2)}$ with $\mathbf{v}^{(2)}|_{\p \Omega}=g_2$, one can integrate the above identity to obtain 
$$
\int_{\Omega}\LC a_1^{(2)}-a_2^{(2)} \RC\mathbf{v}^{(1)}\mathbf{v}^{(2)}\, dx=0.
$$
By applying the density of CGOs again shows that 
\begin{align}\label{uniqueness of quad 2}
	a^{(2)}:=a_1^{(2)}=a_2^{(2)} \text{ in }\Omega.
\end{align}

Let us then define $\psi\in C^2(\overline\Omega)$ as the difference
\begin{align}\label{psi}
  \psi:=u_2^{(0)} -u_1^{(0)} \text{ in }\Omega.
\end{align}
By plugging \eqref{uniqueness of quad 2} into \eqref{uniqueness of 1st coef}, we obtain 
\begin{align}\label{uniqueness of quad 2-1}
	a_1^{(1)}=a_2^{(1)} +2a^{(2)}\LC  u_2^{(0)} -u_1^{(0)}\RC =a_2^{(1)} +2a^{(2)} \psi \text{ in }\Omega.
\end{align}
Moreover, with the relation \eqref{psi} at hand, we calculate
\begin{align}\label{gauge id in pf thm 2}
	\begin{split}
		F_2  = &\Delta  u_2^{(0)} + a_1^{(2)} u_2^{(0)} +  a_2^{(2)}\LC u_2^{(0)}\RC^2 \\
		= &\Delta \LC u_1^{(0)} +\psi  \RC + a_1^{(2)}\LC u_1^{(0)} +\psi  \RC + a_2^{(2)} \LC u_1^{(0)} +\psi  \RC^2  \\
		=& \LC F_1 +a_1^{(2)}\psi +a_2^{(2)}\psi^2 \RC + \LC  a_1^{(2)} -a_1^{(1)}+2a_2^{(2)}\psi   \RC u_1^{(0)}  \\
		&+ \LC a_2^{(2)} -a_2^{(1)} \RC \LC u_1^{(0)}\RC^2.
	\end{split}
\end{align}
Here we also utilized \eqref{equ: zero order linear zero}. By using \eqref{uniqueness of quad 2} and \eqref{uniqueness of quad 2-1}, we see that 
$F_2=F_1 +a_1^{(2)}\psi +a_2^{(2)}\psi^2$ in $\Omega$. Finally, the function $\psi$ of the form \eqref{psi} satisfies $\psi|_{\p \Omega}=\left. \LC u_2^{(0)}-u_1^{(0)}\RC\right|_{\p \Omega}=0$ and  $\left. \p_\nu \psi \right|_{\p \Omega}= \p_\nu \big( u_2^{(0)}-u_1^{(0)}\big) \big|_{\p \Omega}=0$. We have shown
\begin{align}\label{gauge 1_proof}
	\begin{cases}
		a_1^{(2)} =a_2^{(2)}=:a^{(2)} &\\ 
		a_1^{(1)}=a_2^{(1)}+2a^{(2)} \psi &\\ 
		F_1=F_2-\Delta\psi -a_1^{(2)}\psi -a^{(2)}\psi^2& \\
	\end{cases}
\end{align}
as desired.
\end{proof}

\begin{rmk}\label{rem:Runge}
	Note that if the coefficients of quadratic terms vanish, $a_2^{(1)}=a_2^{(2)}=0$ in $\Omega$, then \eqref{gauge 1_proof} describes the gauge symmetry of inverse source problem for linear equation discussed in  Remark \ref{rmk:counterexample}.
	
	We also remark that in the above proof we could have alternatively used Runge approximation argument to show that $a_1^{(2)}=a_2^{(2)}$ after \eqref{eq:for_Runge}. Indeed, if $x_0\in \Omega$, there is by Runge approximation (see e.g. \cite{LLS2020poisson}) a solution $\mathbf{v}^{(1)}$ such that $\mathbf{v}^{(1)}(x_0)\neq 0$. Together with \eqref{eq:for_Runge}, and using the above argument for all $x_0\in \Omega$, shows $a_1^{(2)}=a_2^{(2)}$ in $\Omega$. Runge approximation in similar situations were earlier used in \cite{LLLS2019nonlinear}. 
\end{rmk}

As discussed in the introduction, if the linear term of a semilinear equation $\Delta u + a^{(1)}u +a^{(2)}u^2=F$ is known (i.e., $a^{(1)}$ is known a priori), then the DN map determines the other coefficients of the equation uniquely. This is Corollary \ref{Cor: Unique determination}, which we now prove.
\begin{proof}[Proof of Corollary \ref{Cor: Unique determination}]

 By assumption and Theorem \ref{Thm: gauge with quadratic} 
 $$
 a^{(1)}=a^{(1)}+2a^{(2)}\psi 
 $$
 and
 \[
  F_1=F_2-\Delta\psi -a_1^{(2)}\psi -a^{(2)}\psi^2 
 \]
 hold in $\Omega$ for some gauge function $\psi$. Here $a^{(2)}=a_1^{(2)}=a_2^{(2)}$. Since $a^{(2)}\neq 0$ in $\Omega$ by assumption, the first identity above shows that $\psi=0$ in $\Omega$. Substituting $\psi=0$ to latter identity above shows $F_1=F_2$ in $\Omega$.
\end{proof}

\subsection{Cubic nonlinearity}
We move on to prove our results about cubic nonlinearities. For $j=1,2$, we let  
$$
a_j(x,z)=a_j^{(1)}z +a_j^{(2)}z^2 +a_j^{(3)}z^3,
$$ 
and let us consider the equation
\begin{align}\label{equ formal cubic}
	\Delta  u_j +a_j^{(1)}u_j +a_j^{(2)}u_j^2 +a_j^{(3)}u_j^3=F_j \text{ in }\Omega.
\end{align}
Theorem \ref{Thm: gauge with cubic}, which we prove in this section shows that the inverse source problems of the above equation has uniqueness property for both coefficients and source up to a gauge. 

Before proving Theorem \ref{Thm: gauge with cubic}, let us derive the gauge of the inverse problem. Assume that $u_1$ solves \eqref{equ formal cubic} with boundary value $u_1|_{\p \Omega}=f$. If $\psi\in C^{2}(\overline{\Omega})$, we denote by $a_2^{(1)}$, $a_2^{(2)}$, $a_3^{(2)}$ and $F_2$ another set of coefficients and a source, which may depend on $\psi$. If we denote $u_2=u_1+\psi$, then we have the chain of equivalences 
\begin{align*}
	&\Delta u_2+ a_2^{(1)}u_2+ a_2^{(2)} \LC u_2 \RC^2 +a_2^{(3	)}\LC u_2 \RC^3= F_2 \\ 
	\iff & \Delta \LC u_1 +\psi \RC  + a_2^{(1)} \LC u_1+\psi \RC + a_2^{(2)} \LC u_1 +\psi \RC^2  +a_2^{(3)}\LC u_1+\psi \RC^3= F_2  \\
	\iff & \Delta u_1 +\Delta \psi +a_2^{(1)} u_1 +a_2^{(1)}\psi +a_2^{(2)}\LC u_1 \RC^2 +2a_2^{(2)}\psi u_1 +a_2^{(2)}\psi^2 \\
	& \qquad+a_2^{(3)}\LC u_1^3 +3u_1^2 \psi+3u_1\psi^2 +\psi^3 \RC =F_2, 
\end{align*}
which holds in $\Omega$. 
By using $\Delta u_1 =-a_1^{(1)}u_1-a_1^{(2)}\LC u_1\RC^2 -a_1^{(3)}\LC u_1\RC^3 +F_1$ in $\Omega$ and equating the powers of $u$ gives the following system 
\begin{align}\label{gauge cubic}
	\begin{cases}
		F_1= F_2 -\Delta \psi- a_2^{(1)} \psi-a_2^{(2)} \psi^2-a_2^{(3)}\psi^3  \\ 
		a_1^{(1)}= a_2^{(1)}  + 2 a_2^{(2)} \psi +3a_2^{(3)}\psi^2  \\
		a_1^{(2)}=a_2^{(2)} +3a_2^{(3)}\psi & \\
		a_1^{(3)}=a_2^{(3)}.  
	\end{cases}
\end{align}
The above system of equations describes the gauge invariance for the inverse source problem for cubic nonlinearity. If $\psi|_{\Omega}=\p_\nu \psi|_{\p \Omega}=0$, the above computation shows that corresponding DN maps $\Lambda_{a_1,F_1}$ and $\Lambda_{a_2,F_2}$ are the same. It is impossible to uniquely determine the coefficients and sources from the DN map at the same time. There is a gauge symmetry given by \eqref{gauge cubic}.

We next prove Theorem \ref{Thm: gauge with cubic}, which states that the DN map determines the coefficients and source up to the gauge symmetry \eqref{gauge cubic}.

\begin{proof}[Proof of Theorem \ref{Thm: gauge with cubic}]
	Let us consider the Dirichlet data 
	\[
	f=f(x;\eps)=f_0+\eps_1 f_1 +\eps_2f_2 +\eps_3 f_3 \quad  \text{ on }\quad \p \Omega,
	\]
	where the parameters $\eps_\ell$ are real numbers, $f_0\in \mathcal{N}$ and $f_\ell\in C^{2,\alpha}(\p \Omega)$, for $\ell =1,2,3$. By assumption $\Lambda_{a_1,F_1}(f)=\Lambda_{a_2,F_2}(f)$ if the parameters $\eps_\ell$ are small enough. We denote $\eps=(\eps_1, \eps_2,\eps_3)$. 
	
	Let us denote by $u_j^{(0)}$ the solution to 
	\begin{align*}
 		\begin{cases}
 			\Delta u_j^{(0)}+ a_j(x,u_j^{(0)}) = F_j & \text{ in }\Omega,\\
 			u_j^{(0)}=f_0 &\text{ on }\p \Omega.
 		\end{cases}
 	\end{align*}
	We linearize 
	\begin{align*}
 		\begin{cases}
 			\Delta u_j+ a_j(x,u_j) = F_j & \text{ in }\Omega,\\
 			u_j=f_0+\eps_1 f_1 +\eps_2f_2 +\eps_3 f_3 &\text{ on }\p \Omega.
 		\end{cases}
 	\end{align*}
	at the solution corresponding to boundary value $f_0$ for $j=1,2$. The first linearization at $f_0$ is 
\begin{equation}\label{1st_in_cubic}
	\begin{cases}
		\LC \Delta + a_j^{(1)}+2a_j^{(2)} u_j^{(0)} +3 a_j^{(3)}\LC u_j^{(0)} \RC ^2\RC v_j^{(\ell)} =0 &\text{ in }\Omega, \\
		v_j^{(\ell)}=f_\ell &\text{ on }\p \Omega,
	\end{cases}
\end{equation}
  where $v_j^{(\ell)}:=\left. \p_{\eps_\ell}u_j \right|_{\eps=0}$ in $\Omega$.  By Theorem \ref{Thm:wellposedness_and_expansion}, we know that the DN maps of \eqref{1st_in_cubic} for $j=1$ and $j=2$ agree.  By the global uniqueness result for the Calder\'on problem for linear equations we have 
\begin{align}\label{Q in cubic}
	Q:=a_1^{(1)}+2a_1^{(2)}u_1^{(0)} +3 a_1^{(3)} \LC u_1^{(0)}\RC^2=a_2^{(1)}+2a_2^{(2)}u_2^{(0)} +3 a_2^{(3)} \LC u_2^{(0)}\RC^2 \text{ in }\Omega,
\end{align}
and by the uniqueness of solutions to the Dirichlet problem \eqref{1st_in_cubic} it follows that
\begin{align}\label{v ell in cubic}
	v^{(\ell)}:=v_1^{(\ell)}=v_2^{(\ell)} \text{ in }\Omega,
\end{align}
for $\ell =1,2,3$. 

The second linearization reads
 \begin{align}\label{2nd in cubic}
	\begin{cases}
		\LC \Delta + Q\RC w_j^{(k\ell)}   + 2\LC a_j^{(2)} +3 a_j^{(3)}u_j^{(0)} \RC v^{(k)}v^{(\ell)} =0 &\text{ in }\Omega,\\
		w_j^{(k\ell)}=0 &\text{ on }\p \Omega,
	\end{cases}
\end{align}
where $w_j^{(k\ell)}=\left. \p^2_{\eps_k \eps_\ell}u_j \right|_{\eps=0}$ for $k,\ell \in \{1,2,3\}$ and $j=1,2$. Similar to the proof of Theorem \ref{Thm: gauge with quadratic}, multiplying \eqref{2nd in cubic} by the function $\mathbf{v}$ that solves
\begin{align}\label{auxil 2nd in cubic}
	\begin{cases}
			\LC \Delta + Q\RC  \mathbf{v}=0 &\text{ in }\Omega, \\
			\mathbf{v}=\mathbf{g}&\text{ on }\p \Omega,
	\end{cases}
\end{align}
where $\mathbf{g}\in H^{1/2}(\p \Omega)$ is a function to be chosen later. 
Multiplying \eqref{2nd in cubic} by the solution $\mathbf{v}$ and integrating by parts show that 
\begin{align}\label{2nd inte id in cubic}
	\int_\Omega \left[ \LC a_1^{(2)}+3a_1^{(3)}  u_1^{(0)} \RC -\LC a_2^{(2)}+3a_2^{(3)} u_2^{(0)} \RC \right] v^{(k)}v^{(\ell)}\mathbf{v}\,  dx=0,
\end{align} 
for $k,\ell =1,2,3$.
Applying an additional density argument as in the proof of Theorem \ref{Thm: gauge with quadratic} (or the one described in Remark \ref{rem:Runge}), one obtains
\begin{align}\label{2nd lin id in cubic}
	R:=a_1^{(2)}+3a_1^{(3)} u_1^{(0)}  =a_2^{(2)}+3a_2^{(3)}  u_2^{(0)}  \text{ in }\Omega.
\end{align}
The uniqueness of solutions to Dirichlet problem of \eqref{2nd in cubic} and \eqref{2nd lin id in cubic} imply 
\[
w^{(k\ell)}:=w_1^{(k\ell)}=w_2^{(k\ell)} \text{ in }\Omega,
\]
for any $k, \ell \in \{1,2,3\}$.

Now, a computation shows that the third linearized equation is 
\begin{align}\label{3rd in cubic}
	\begin{cases}
		\LC \Delta +Q\RC w_j^{(123)} +2R \LC w^{(12)}v^{(3)}+w^{(23)}v^{(1)}+w^{(13)}v^{(2)}\RC \\
		\qquad \qquad \qquad \quad +6a_j^{(3)} v^{(1)}v^{(2)}v^{(3)}=0 &\text{ in }\Omega, \\
		w^{(123)}=0 &\text{ on }\p \Omega,
	\end{cases}
\end{align}
where $R$ is the function given by \eqref{2nd lin id in cubic}.
Multiplying \eqref{3rd in cubic} against the solution $\mathbf{v}$ of \eqref{auxil 2nd in cubic} and integrating by parts produces the identity
\begin{align*}
	\int_{\Omega} \LC a_1^{(3)}-a_2^{(3)} \RC v^{(1)}v^{(2)}v^{(3)}\mathbf{v}\, dx=0.
\end{align*}
By choosing $v^{(\ell)}$ ($\ell=1,2,3$) and $\mathbf{v}$ to be suitable CGO solutions, we conclude via the above integral identity  
\begin{align}\label{unique a_3 in cubic}
	a^{(3)}:=a_1^{(3)}=a_2^{(3)} \text{ in }\Omega, 
\end{align} 
which proves the first relation in \eqref{gauge cubic thm}.

Let us define $\psi\in C^2(\overline\Omega)$ by 
\begin{align}\label{psi in cubic}
	\psi=u_2^{(0)} -u_1^{(0)} \text{ in }\Omega.
\end{align}
Then, the identity \eqref{2nd lin id in cubic} is equivalent to 
\begin{align}\label{unique gange 2nd in cubic}
	a_1^{(2)}=a_2^{(2)} + 3a^{(3)} \LC u_2^{(0)}-u_1^{(0)} \RC = a_2^{(2)} + 3a^{(3)} \psi, 
\end{align}
where we utilized \eqref{unique a_3 in cubic} and \eqref{psi in cubic}. This shows the second identity in \eqref{gauge cubic thm}.
By plugging \eqref{unique gange 2nd in cubic} into \eqref{Q in cubic}, direct computations yield
\begin{align*}
	a_1^{(1)}=&a_2^{(1)}+2a_2^{(2)}u_2^{(0)} +3 a^{(3)} \LC u_2^{(0)}\RC^2-2a_1^{(2)}u_1^{(0)} -3 a^{(3)} \LC u_1^{(0)}\RC^2 \\
	=& a_2^{(1)}+2a_2^{(2)}u_2^{(0)} +3 a^{(3)} \LC u_2^{(0)}\RC^2-2a_2^{(2)}u_1^{(0)}-6a^{(3)}\psi  u_1^{(0)}-3 a^{(3)} \LC u_1^{(0)}\RC^2 \\
	=& a_2^{(1)} +2a_2^{(2)} \psi +3a^{(3)}  \psi^2,
\end{align*}
which proves the third identity in \eqref{gauge cubic thm}. Finally, by inserting \eqref{psi in cubic} into the original nonlinear equation \eqref{equation in thm cubic}, and equating the powers of $u_2^{(0)}$, yield the last identity in \eqref{gauge cubic thm} as desired. This completes the proof.
\end{proof}

\subsection{Polynomial and general nonlinearity}
 
In order to prove Theorem \ref{Thm: gauge polynomial}, where the nonlinearity is a general polynomial, it is convenient to prove Theorem \ref{Thm: general nonlinearity} about general nonlinearities first.

 \begin{proof}[Proof of Theorem \ref{Thm: general nonlinearity}]
 	Let $N\in \N$. By using the higher order linearization method, let us take the Dirichlet data to be of the form 
 	\[
 	f(x)=\sum_{\ell=1}^N \eps_\ell f_\ell (x), \quad x \text{ in }\p \Omega,
 	\]
 	where $\eps_\ell$ are parameters such that $|\eps_\ell|$ are sufficiently small, and each $f_\ell \in C^{2,\alpha}(\p \Omega)$, for $\ell=1,\ldots, N$.
 	We first linearize the equation \eqref{semilinear ellip index} around the solution $u_j^{(0)}$, so that we can have 
 	\begin{align}\label{1st lin gener}
 		\begin{cases}
 			\LC \Delta + \p_za_j(x,u_j^{(0)}) \RC v_j^{(\ell)} = 0 &\text{ in }\Omega, \\
 			v_j^{(\ell)}=f_\ell &\text{ on }\p \Omega
 		\end{cases}
 	\end{align}
 	for $j=1,2$, and $\ell =1,\ldots,N$. The uniqueness result for the inverse problem for the linear Shr\"odinger equation yields again that 
 	\[
 	\p_z a_1(x,u_1^{(0)})=\p_z a_2(x,u_2^{(0)}) \text{ in }\Omega.
 	\]
 	Moreover, via the uniqueness of solutions, we have $v^{(\ell)}=v_1^{(\ell)}=v_2^{(\ell)}$ in $\Omega$, for $\ell=1,2,\ldots,N$.

 	To proceed, the second linearized equation can be derived as 
 	\begin{align}\label{2nd lin gener}
 		\begin{cases}
 			\LC \Delta + Q \RC w_j^{(\ell m)} + \p^2_z a_j (x,u_j^{(0)}) v^{(\ell)} v^{(m)} = 0 &\text{ in }\Omega, \\
 			w_j^{(\ell m)}=0 &\text{ on }\p \Omega,
 		\end{cases}
 	\end{align}
 	where $Q:=\p_za_1(x,u_1^{(0)})=\p_za_2(x,u_2^{(0)})$ in $\Omega$, for $\ell,m=1,2,\ldots, N$. Similar as before, consider a solution $\mathbf{v}$ of 
 	\begin{align*}
 		\begin{cases}
 		\LC \Delta + Q \RC   \mathbf{v} =0 &\text{ in }\Omega,\\
 		\mathbf{v}=\mathbf{g} &\text{ on }\p \Omega,
 		\end{cases}
 	\end{align*}
    by multiplying \eqref{2nd lin gener} by the function $\mathbf{v}$, then an integration by parts formula yields that 
    \begin{align*}
    	\int_{\Omega} \LC \p_z^2a_1(x,u_1^{(0)})-\p_z^2a_2(x,u_2^{(0)}) \RC v^{(\ell)} v^{(m)}  \mathbf{v}\, dx=0,
    \end{align*}
which shows $\p_z^2a_1(x,u_1^{(0)})=\p_z^2a_2(x,u_2^{(0)}) $ in $\Omega$ by utilizing preceding arguments.
 	
 	Furthermore, by considering higher order linearized equations and using an induction argument, similar to the ones in the proofs of \cite[Proof of Theorem 1.1]{LLLS2019partial} and \cite[Proof of Theorem 1.3]{KU2019partial}, it is not hard to show that \eqref{same Taylor coef} holds for any $k\in \N$, where $u_j^{(0)}$ are the solutions of \eqref{1st lin gener}, for $j=1,2$. As $N\in \N$ was arbitrary, this completes the proof. 
 	 \end{proof}
 
We now prove Theorem \ref{Thm: gauge polynomial}.

\begin{proof}[Proof of Theorem \ref{Thm: gauge polynomial}]
To prove the theorem, we need to show that there is $\psi\in C^{2,\alpha}(\overline \Omega)$ with $\psi|_{\p \Omega}=\p_\nu \psi|_{\p \Omega}=0$ such that
\begin{align}\label{gauge poly_proof}
		a_1^{(N-k)}=\sum_{m=N-k}^N  \LC \begin{matrix}
			m\\
			N-k\\\end{matrix}  \RC a_2^{(m)}\psi ^{m-N+k} \quad \text{ in }\Omega,
	\end{align}
for $k=1,\ldots, N$. Since $a_1(x,z)$ and $a_2(x,z)$ are both polynomials of order $N$, we have by Theorem \ref{Thm: general nonlinearity}
\[
 a_1^{(N)}(x)=\p_z^Na_1(x,u_1^{(0)})=\p_z^Na_2(x,u_2^{(0)})=a_2^{(N)}(x)
\]
for all  $x\in \Omega$. Here $u_j^{(0)}$, $j=1,2$, is the solution of \eqref{eq poly gauge inva} as $u_j^{(0)} \big|_{\p \Omega}=0$. Thus the claim holds for $k=0$. We prove the claim by induction. For this, let us assume that 
	\eqref{gauge poly_proof} holds for all $k=0,\ldots, L$. It suffices to show that \eqref{gauge poly} holds for $k=L+1$.
	
	Using Theorem \ref{Thm: general nonlinearity} again, we have 
    \begin{align}\label{N-L-1 derivative}
    	\p_z ^{N-(L+1)}a_1(x,u_1^{(0)})=\p _z ^{N-(L+1)}a_2(x,u_2^{(0)}) \quad \text{ in }\Omega.
    \end{align}
    Since $a_j(x,z)$ is a polynomial in $a$, this identity is equivalent to 
	 \begin{align*}
	 	&(N-L-1)! a_1^{(N-L-1)}+(N-L)! a_1^{(N-L)}u_1^{(0)} \\
	 	&\quad + \frac{(N-L+1)!}{2!}a_1^{(N-L+1)} \LC  u_1^{(0)}\RC^2+\cdots +\frac{N!}{(L+1)!}a_1^{(N)}\LC u_1^{(0)}\RC^{L+1} \\
	 	=&(N-L-1)! a_2^{(N-L-1)}+(N-L)! a_2^{(N-L)}u_2^{(0)} \\
	 	&\quad + \frac{(N-L+1)!}{2!}a_2^{(N-L+1)} \LC  u_2^{(0)}\RC^2+\cdots +\frac{N!}{(L+1)!}a_2^{(N)}\LC u_2^{(0)}\RC^{L+1}.
	 \end{align*}
    After diving by $(N-L-1)!$ the above reads
    \begin{align}\label{id poly 1}
    	\begin{split}
    		&	\LC \begin{matrix}
    			N-L-1\\
    			N-L-1\\\end{matrix}  \RC a_1^{(N-L-1)}+\LC \begin{matrix}
    			N-L\\
    			N-L-1\\\end{matrix}  \RC a_1^{(N-L)}u_1^{(0)} \\
    		&\quad + \LC \begin{matrix}
    			N-L+1\\
    			N-L-1\\\end{matrix}  \RC  a_1^{(N-L+1)} \LC  u_1^{(0)}\RC^2+\cdots +\LC \begin{matrix}
    			N\\
    			N-L-1\\\end{matrix}  \RC  a_1^{(N)}\LC u_1^{(0)}\RC^{L+1} \\
    		=&\LC \begin{matrix}
    			N-L-1\\
    			N-L-1\\\end{matrix}  \RC a_2^{(N-L-1)}+\LC \begin{matrix}
    			N-L\\
    			N-L-1\\\end{matrix}  \RC  a_2^{(N-L)}u_2^{(0)} \\
    		&\quad + \LC \begin{matrix}
    			N-L+1\\
    			N-L-1\\\end{matrix}  \RC  a_2^{(N-L+1)} \LC  u_2^{(0)}\RC^2+\cdots +\LC \begin{matrix}
    			N\\
    			N-L-1\\\end{matrix}  \RC  a_2^{(N)}\LC u_2^{(0)}\RC^{L+1}.
    	\end{split}
    \end{align}
We rewrite \eqref{id poly 1} as
\begin{align}\label{id poly 2}
	\begin{split}
	&	a_1^{(N-L-1)} +\sum_{k=0}^{L}  \LC \begin{matrix}
			N-L+k\\
			N-L-1\\\end{matrix}  \RC a_1^{(N-L+k)} \LC u_1^{(0)}\RC^{k+1} \\
		=& a_2^{(N-L-1)} +\sum_{k=0}^{L}  \LC \begin{matrix}
			N-L+k\\
			N-L-1\\\end{matrix}  \RC a_2^{(N-L+k)} \LC u_2^{(0)}\RC^{k+1}.
	\end{split}
\end{align}

We define
\begin{align}\label{psi poly}
	\psi:= u_2^{(0)}-u_1^{(0)}.
\end{align}
Then $\psi\in C^{2,\alpha}(\Omega)$ and $\psi|_{\p \Omega}=\p_\nu \psi|_{\p \Omega}=0$.
%
By using the induction assumption, that \eqref{gauge poly_proof} holds for $k=0,\ldots,L$, we write the identity \eqref{id poly 2} as
\begin{align}\label{id poly 4}
	\begin{split}
		&a_1^{(N-L-1)} \\
		&= a_2^{(N-L-1)} +\sum_{k=0}^L \LC \begin{matrix}
			N-L+k\\
			N-L-1\\\end{matrix}  \RC \left[ a_2^{(N-L+k)} \LC u_2^{(0)}\RC^{k+1} -  a_1^{(N-L+k)} \LC u_1^{(0)}\RC^{k+1} \right] \\
		&=a_2^{(N-L-1)}+\sum_{k=0}^L \LC \begin{matrix}
			N-L+k\\
			N-L-1\\\end{matrix}  \RC \Bigg[ a_2^{(N-L+k)} \LC u_1^{(0)}+\psi \RC^{k+1} \\
		& \qquad\qquad\qquad\qquad\qquad\quad  -  \sum_{m=N-L+k}^N  \LC \begin{matrix}
			m\\
			N-L+k\\\end{matrix}  \RC a_2^{(m)}\psi ^{m-N+L-k}\LC u_1^{(0)}\RC^{k+1} \Bigg] \\
	\end{split}
\end{align}
Here the induction assumption was used in the last equality. By using binomial expansion, the above equality is
\begin{align}\label{id poly 3}
	\begin{split}
		&a_1^{(N-L-1)} \\
		&=a_2^{(N-L-1)}+\sum_{k=0}^L \LC \begin{matrix}
			N-L+k\\
			N-L-1\\\end{matrix}  \RC \Bigg[ a_2^{(N-L+k)} \sum_{\iota =0}^{k+1}\LC \begin{matrix}
			k+1\\
			\iota \\\end{matrix}  \RC  \psi^{\iota}  \LC u_1^{(0)} \RC^{k+1-\iota}    \\
		& \qquad \qquad\qquad\qquad\qquad -  \sum_{m=N-L+k}^N  \LC \begin{matrix}
			m\\
			N-L+k\\\end{matrix}  \RC a_2^{(m)}\psi ^{m-N+L-k}\LC u_1^{(0)}\RC^{k+1} \Bigg]\\
	   &=  a_2^{(N-L-1)}+ S_1 - S_2.
	\end{split}
\end{align}
%
Here we have defined
\begin{align}\label{I-II}
	\begin{split}
		S_1:= &\sum_{k=0}^L \LC \begin{matrix}
			N-L+k\\
			N-L-1\\\end{matrix}  \RC  a_2^{(N-L+k)} \sum_{\iota =0}^{k+1}\LC \begin{matrix}
			k+1\\
			\iota \\\end{matrix}  \RC  \psi^{\iota}  \LC u_1^{(0)} \RC^{k+1-\iota}    ,\\
		S_2:= &  \sum_{k=0}^L \LC \begin{matrix}
			N-L+k\\
			N-L-1\\\end{matrix}  \RC \sum_{m=N-L+k}^N  \LC \begin{matrix}
			m\\
			N-L+k\\\end{matrix}  \RC a_2^{(m)}\psi ^{m-N+L-k}\LC u_1^{(0)}\RC^{k+1}.
	\end{split}
\end{align}

To complete the proof we compare the coefficients of the powers of $u_1^{(0)}$ of $S_1$ and $S_2$. We first observe that in the term $S_1$, the powers of $u_1^{(0)}$ range from $0$ to $L+1$. In the tern $S_2$, the powers of $ u_1^{(0)}$ range from $1$ to $L+1$. We split the remaining proof into two cases according to powers of $u_1^{(0)}$. 

\medskip

\noindent {\bf Case 1:}
\medskip

\noindent Let us consider the coefficients of the terms $\big(u_1^{(0)} \big)^J$, $J=1,\ldots, L+1$, in $S_1$ and $S_2$. We observe that the coefficient of $\big( u_1^{(0)} \big)^J$ in $S_1$ is 
\begin{align}\label{I_J}
	\sum_{k=J-1}^L  \LC \begin{matrix}
		N-L+k\\
		N-L-1\\\end{matrix}  \RC a_2^{(N-L+k)}\LC \begin{matrix}
		k+1\\
		k+1-J \\\end{matrix}  \RC  \psi^{k+1-J} 
\end{align}
Similarly, the coefficient of $\big( u_1^{(0)} \big)^J$ in $S_2$ is  
\begin{align}\label{II_J}
	\begin{split}
		&\LC \begin{matrix}
			N-L+J+1\\
			N-L-1\\\end{matrix}  \RC \sum_{m=N-L+J-1}^N  \LC \begin{matrix}
			m\\
			N-L+J+1\\\end{matrix}  \RC a_2^{(m)}\psi ^{m-N+L-J+1} \\
		=&\LC \begin{matrix}
			N-L+J-1\\
			N-L-1\\\end{matrix}  \RC \sum_{k=J-1}^{L}  \LC \begin{matrix}
			N-L+k\\
			N-L+J-1\\\end{matrix}  \RC  a_2^{(N-L+k)} \psi^{k+1-J}.
	\end{split}
\end{align}
On the other hand, a direct computation shows that 
\begin{align*}
\LC \begin{matrix}
	N-L+k\\
	N-L-1\\\end{matrix}  \RC \LC \begin{matrix}
		k+1\\
		k+1-J \\\end{matrix}  \RC =\LC \begin{matrix}
		N-L+J-1\\
		N-L-1\\\end{matrix}  \RC\LC \begin{matrix}
		N-L+k\\
		N-L+J-1\\\end{matrix}  \RC ,
\end{align*} 
so that \eqref{I_J} and \eqref{II_J} are the same.

\medskip

\noindent{\bf Case 2:} 

\medskip

\noindent The term $S_2$ does not contain the zeroth power of $u_1^{(0)}$. We express $S_1$ as 
\[
S_1:=S_0+\wt S,
\]
where 
\begin{align}\label{I_0-tilde I}
	\begin{split}
		S_0:=&\sum_{k=0}^L  \LC \begin{matrix}
			N-L+k\\
			N-L-1\\\end{matrix}  \RC  a_2^{(N-L+k)} \psi ^{k+1} \\
		\wt S:=&\sum_{k=0}^L \LC \begin{matrix}
			N-L+k\\
			N-L-1\\\end{matrix}  \RC  a_2^{(N-L+k)} \sum_{\iota =0}^{k}\LC \begin{matrix}
			k+1\\
			\iota \\\end{matrix}  \RC  \psi^{\iota}  \LC u_1^{(0)} \RC^{k+1-\iota}.
	\end{split}
\end{align}
By redefining the summation index of $S_0$, we have 
\begin{align}\label{I_0}
	\begin{split}
		S_0=&\sum_{m=N-L}^N  \LC \begin{matrix}
			m\\
			N-L-1\\\end{matrix}  \RC  a_2^{(m)} \psi ^{m-N+L+1} .
	\end{split}
\end{align}
Therefore, by plugging \eqref{I_J}--\eqref{I_0} into \eqref{id poly 4}, we obtain 
\begin{align*}
	a_1^{(N-(L+1))}=&a_2^{(N-(L+1))} +\sum_{m=N-L}^N  \LC \begin{matrix}
		m\\
		N-L-1\\\end{matrix}  \RC  a_2^{(m)} \psi ^{m-N+L+1}  \\
=&\sum_{m=N-(L+1)}^N  \LC \begin{matrix}
	m\\
	N-L-1\\\end{matrix}  \RC  a_2^{(m)} \psi ^{m-N+L+1}.
\end{align*}
This proves the induction step. It remains to prove \eqref{F_rel_poly}.

Recall that the nonlinearity $a_j(x,z)=\sum_{k=1}^N a_j^{(k)}z^k$, for $j=1,2$, then we can write 
$a_1(x,u_1^{(0)})$ in terms of 
\begin{align}\label{a_1 in poly}
	\begin{split}
		a_1(x,u_1^{(0)})=&\sum_{k=0}^{N-1} a_1^{(N-k)} \LC u_1^{(0)} \RC^{N-k}\\
		=&\sum_{k=0}^{N-1}\sum_{m=N-k}^N  \LC \begin{matrix}
			m\\
			N-k\\\end{matrix}  \RC a_2^{(m)}\psi ^{m-N+k}\LC u_1^{(0)} \RC^{N-k}.
	\end{split}
\end{align}
On the other hand, one can also express 
\begin{align}\label{a_2 in poly}
	\begin{split}
		a_2(x,u_2^{(0)})=\sum_{k=1}^N a_2^{(k)}\LC u_2^{(0)}\RC^k =\sum_{k=1}^N a_2^{(k)}  \sum_{m=0}^k\LC \begin{matrix}
			k\\
			m\\\end{matrix}  \RC \LC u_1^{(0)} \RC^m \psi^{k-m},
	\end{split}
\end{align}
where we used \eqref{psi poly} and binomial expansion in the above computation.  Similar to the computations of Case 1 in preceding arguments, by comparing the orders of the homogeneous parts $\big( u_1^{(0)} \big)^{L}$, for $L=1,2,\ldots, N$, a direct computation yields that 
\[
a_2(x,u_2^{(0)})-a_1(x,u_1^{(0)})=\sum_{k=1}^N a_2^{(k)}   \psi^{k}.
\]
Therefore, 
\begin{align*}
	\begin{split}
		F_1 -F_2=\Delta \LC u_1^{(0)} -u_2^{(0)} \RC +a_1(x,u_1^{(0)}) -a_2(x,u_2^{(0)})
		=- \Delta \psi -\sum_{k=1}^N a_2^{(k)}   \psi^{k},
	\end{split}
\end{align*}
which shows \eqref{F_rel_poly}. This proves the assertion.
\end{proof}

With Theorem \ref{Thm: gauge polynomial} at hand, we can prove Corollary \ref{Cor: Unique determination poly} immediately.

\begin{proof}[Proof of Corollary \ref{Cor: Unique determination poly}]
With the identities \eqref{gauge poly} at hand, as $k=1$, we have
	\begin{align*}
	a_1^{(N-1)}=\sum_{m=N-1}^N  \LC \begin{matrix}
		m\\
		N-1\\\end{matrix}  \RC a_2^{(m)}\psi ^{m-N+1} =a_2^{(N-1)} + N a_2^{(N)}\psi  \text{ in }\Omega.
\end{align*}
Since $a_1^{(N)}(x)=a_2^{(N)}(x)\neq 0 $ for all $x\in \Omega$, and $a_1^{(N-1)}=a_2^{(N-1)}$, the preceding equality yields that $\psi=0$ in $\Omega$. Finally, by applying the \eqref{gauge poly} again as $k=N$, one can prove $F_1=F_2$ in $\Omega$, which completes the proof.
\end{proof}

	We next prove that is the  sources $F_1$ and $F_2$ are known in Theorems \ref{Thm: gauge with quadratic}--\ref{Thm: gauge polynomial}, then it is possible to determine the coefficients uniquely. We have the following corollary, which we formulate in terms of the general polynomial nonlinearity. 
	
\begin{cor}
	Let us adopt the notation and assumptions in Theorem \ref{Thm: gauge polynomial}. If $F_1=F_2$ in $\Omega$, then we have 
	\begin{align*}
		a_1^{(k)}=a_2^{(k)} \text{ in }\Omega,
	\end{align*}
for $k=1,2,\ldots ,N$.
\end{cor}

\begin{proof}
	By using  \eqref{F_rel_poly}, we have 
	\begin{align}\label{poly psi}
		\Delta \psi +\sum_{k=1}^N a_2^{(k)}\psi^k =0 \text{ in }\Omega,
	\end{align}
	where $\psi \in C^{2,\alpha}(\overline{\Omega})$ is defined via \eqref{psi poly}, which is a bounded function. Since $a_2^{(k)}\in  C^{\alpha}(\overline{\Omega})$ for $k=1,2,\ldots, N$, \eqref{poly psi} implies that 
	\begin{align*}
	   \begin{cases}
	   		|\Delta \psi |\leq C|\psi| &\text{ in }\Omega,\\
	   		\psi =\p_\nu \psi =0 &\text{ on }\p \Omega,
	   \end{cases}
	\end{align*}
for some constant $C>0$.
Applying the unique continuation for differential inequalities (see e.g.~\cite{JK1985unique}), one obtains that $\psi =0$ in $\Omega$. Finally, combining with the relations \eqref{gauge poly}, we obtain the uniqueness of coefficients. (To easily see how this final argument goes, see the cubic case and \eqref{gauge cubic} first.) 
\end{proof}

\section{Case studies of Theorem \ref{Thm: general nonlinearity}}\label{Sec 4}

In the end of this paper, we study special cases Theorem \ref{Thm: general nonlinearity}, which stated that for general nonlinearities 
\begin{align}\label{same Taylor coef_case_study}
		\p_z^k a_1(x,u_1^{(0)}(x))= \p_z^k a_2(x,u_2^{(0)}(x)), \quad  x\in \Omega, \quad k\in \N.
\end{align}
In general, given only the conditions \eqref{same Taylor coef_case_study}, it is not clear how  explicit relation between the coefficients $(a_1(x,z),F_1(x))$ and $(a_2(x,z),F_2(x))$ in terms of $\psi=u_2^{(0)}-u_1^{(0)}$ one can find. This final section of this paper consider examples where the relation is explicit. 


\subsection{Exponential nonlinearity}


\begin{proof}[Proof of Theorem Corollary \ref{Cor: Exponential}]
 We prove cases 1 and 2 separately:
 
 \smallskip
 
 \noindent \textbf{Case 1.}
 
 \smallskip

 The nonlinearity in this case is $a_j(x,z)=q_j(x)e^z$. Let $u_j^{(0)}$ be the solution to 
 \begin{align}\label{u0 in exponential}
 	\begin{cases}
 		\Delta u_j^{(0)} + q_j(x)e^{u_j^{(0)}}=F_j &\text{ in }\Omega, \\
 		u_j^{(0)}=f_0&\text{ on }\p \Omega,
 	\end{cases}
 \end{align}
for $j=1,2$. Here $f_0\in \mathcal{N}$. Using \eqref{same Taylor coef_case_study} with $k=1$, we have  
 \begin{align}\label{q1u1=q2u2}
 	q_1e^{u_1^{(0)}}=\p_z a_1(x,u_1^{(0)})=\p_z a_2(x,u_2^{(0)})=q_2 e^{u_2^{(0)}} \text{ in }\Omega.
 \end{align}
  On the other hand, by taking $u_2^{(0)}=u_1^{(0)}+\psi$ in $\Omega$, by \eqref{q1u1=q2u2} one has  $q_1e^{u_1^{(0)}}=q_2e^{u_1^{(0)}+\psi}$ which implies $q_1=q_2e^{\psi}$ in $\Omega$.
 Then, by using \eqref{u0 in exponential}, we have 
 \[
F_2-F_1= \Delta \big( u_2^{(0)}-u_1^{(0)} \big)  + q_2e^{u_2^{(0)}} - q_1e^{u_1^{(0)}} =\Delta \psi \text{ in }\Omega,
 \]
 where we have utilized \eqref{q1u1=q2u2}. This shows \eqref{gauge exponential}.
 
 For the converse statement, we note that if
 \[
  q_1=q_2e^{\psi} \text{ and } F_1=F_2-\Delta\psi,
 \]
 and we set $u_2=u_1+\psi$, then
 \begin{align*}
  \Delta u_1+ q_1e^{u_1}=F_1 &\iff \Delta u_2 -\Delta \psi+ q_2e^{\psi}e^{u_2-\psi}=F_2-\Delta\psi \\
  &\iff \Delta u_2 + q_2e^{u_2}=F_2.
 \end{align*}
Since $\psi|_{\p \Omega}=\p_\nu|_{\p \Omega}=0$, we have the converse statement.

  \smallskip
 
 \noindent \textbf{Case 2.}
 
 \smallskip
 
 In this case $a_j(x,z)=q_j(x)ze^z$. Let $u_j^{(0)}$ be the solution of 
 \begin{align}\label{equ u eu 0}
 	\begin{cases}
 		\Delta u_j^{(0)} + q_ju_j^{(0)} e^{u_j^{(0)} } =F_j &\text{ in }\Omega, \\
 		u_j^{(0)} =f_0 &\text{ on }\p \Omega,
 	\end{cases}
 \end{align}
 for $j=1,2$. The condition \eqref{same Taylor coef_case_study} for $k=1$ yields
 \begin{align}\label{uniq fir u eu}
 	Q:=q_1 \LC u_1^{(0)} +1 \RC e^{u_1^{(0)}}=	q_2 \LC u_2^{(0)} +1 \RC e^{u_2^{(0)}}  \text{ in }\Omega,
 \end{align}
 and for $k=2$ it yields 
 \begin{align}\label{uniq sec u eu}
 	q_1 \LC u_1^{(0)} +2 \RC e^{u_1^{(0)}} = q_2 \LC u_2^{(0)} +2 \RC e^{u_2^{(0)}}  \text{ in }\Omega.
 \end{align}
 Combining \eqref{uniq fir u eu} and \eqref{uniq sec u eu}, we obtain 
 \begin{align}\label{uni coe u eu}
 	q_1e^{u_1^{(0)}}=q_2e^{u_2^{(0)}}	\quad \text{ and }\quad  q_1u_1^{(0)}e^{u_1^{(0)}}=q_2u_2^{(0)}e^{u_2^{(0)}} 	 \text{ in }\Omega.
 \end{align}
 By the first identity of \eqref{uni coe u eu}, we have $q_1=q_2e^{u_2^{(0)}-u_1^{(0)}}$ in $\Omega$. The second identity of \eqref{uni coe u eu} shows that $q_2u_1^{(0)}e^{u_2^{(0)}}=q_2u_2^{(0)}e^{u_2^{(0)}}$ in $\Omega$. Since $q_2\neq 0$
 in $\Omega$, we must have $u_1^{(0)}=u_2^{(0)}$ in $\Omega$, which implies that $F_1=F_2$ in $\Omega$, where we utilized the equation \eqref{equ u eu 0}. Moreover, by the first identity of \eqref{uni coe u eu} and $u_1^{(0)}=u_2^{(0)}$ in $\Omega$, we can derive $q_1=q_2$ in $\Omega$. This proves the assertion. 
\end{proof}

If $F_1=F_2$ in $\Omega$ in Corollary \ref{Cor: Exponential}, we have the following uniqueness result regarding the Case 1 in the above corollary.

\begin{cor} Let us assume as in the Case 1 of Corollary \ref{Cor: Exponential} and adopt its notation. If additionally $F_1=F_2$, then
\[
	q_1=q_2 \text{ in } \Omega.
\]
\end{cor}

\begin{proof}
	Since the source terms in Corollary \ref{Cor: Exponential} satisfy $F_1=F_2$ in $\Omega$, it follows from \eqref{gauge exponential} that $\Delta \psi=0$ in $\Omega$ with $\psi|_{\p \Omega}=\left. \p_\nu\psi \right|_{\p \Omega}=0$. By using the unique continuation principle, we conclude that 
	$\psi\equiv 0$ in $\Omega$. Therefore, combining with \eqref{q1u1=q2u2}, we must have $q_1=q_2$ in $\Omega$ as desired.
\end{proof}

\subsection{The sine-Gordon equation}
We prove Corollary \ref{Cor:sine-Gordon}.

\begin{proof}[Proof of Corollary \ref{Cor:sine-Gordon}]
	We divide the proof into two steps:
	
	\smallskip
	
	{\it Step 1. Gauge invariance.}
	
	\smallskip
	
	\noindent Let $u_j^{(0)}$ be the solution of 
	\begin{align}\label{sine-Gordon zero}
		\begin{cases}
			\Delta u_j^{(0)} +q_j \sin(u_j^{(0)})=F_j &\text{ in }\Omega, \\
			u_j^{(0)}=f_0&\text{ on }\p \Omega,
		\end{cases}
	\end{align}
 for $j=1,2$ and where $f_0\in \mathcal{N}$. By Theorem \ref{Thm: general nonlinearity}, we have $\p_z^ka_1(x,u_1^{(0)})=\p_z^ka_2(x,u_2^{(0)})$, for $k=1,2$, which implies that 
 \begin{align}\label{equal sine-cosine}
 	\quad q_1 \cos u_1^{(0)}=q_2 \cos u_2^{(0)}  \quad \text{ and } \quad  q_1 \sin u_1^{(0)}=q_2 \sin u_2^{(0)} \text{ in }\Omega.
 \end{align}
By the Euler identity, we have $e^{\mathbf{i}y}=\cos y+\mathbf{i}\sin y$, where $\mathbf{i}=\sqrt{-1}$. Then \eqref{equal sine-cosine} is equivalent to 
\begin{align}\label{euler expression}
	q_1 e^{\mathbf{i}u_1^{(0)}}=q_2e^{\mathbf{i}u_2^{(0)}} \text{ in }\Omega.
\end{align}
By defining $\psi=u_2^{(0)}-u_1^{(0)}$, we have that $\psi\in C^{2,\alpha}(\overline{\Omega})$ and $\psi=\p_\nu\psi=0$ on $\p \Omega$. Via the second identity of \eqref{equal sine-cosine} and \eqref{sine-Gordon zero}, one has
\[
\Delta \psi =\Delta \big( u_2^{(0)}-u_1^{(0)}\big)  =F_2-F_1 \text{ in }\Omega,
\]
and by \eqref{euler expression}, 
\[
q_1e^{\mathbf{i}u_1^{(0)}}=q_2e^{\mathbf{i}(u_1^{(0)}+\psi)} \text{ in }\overline{\Omega}, 
\]
which implies $q_1=q_2e^{\mathbf{i}\psi}$ in $\overline{\Omega}$. Furthermore, since $q_1$ and $q_2$ are real-valued functions and $\psi$ is continuous, we must have either 
$e^{\mathbf{i}\psi}\equiv-1$ or $e^{\mathbf{i}\psi}\equiv 1$ in $\Omega$. Thus
\begin{align}\label{gauge sine pm}
	q_1 =\pm q_2 \text{ in }\overline{\Omega}.
\end{align}
It remains to show that 
\begin{align}\label{psi constant}
	e^{\mathbf{i}\psi}=1 \text{ in }\overline{\Omega}.
\end{align}

	\smallskip

{\it Step 2. Boundary determination.}

\smallskip

\noindent We show by using boundary determination that $\psi\equiv 1$ in $\overline \Omega$. Let $\eps$ be a small real parameter, $g\in C^{2,\alpha}(\p \Omega)$ and $f=f_0+\eps g$. By linearizing \eqref{sine-Gordon thm} around the solution $u_j^{(0)}$ of \eqref{sine-Gordon zero},  one has 
\begin{align}\label{first lin sine-Gordon}
	\begin{cases}
	  \LC \Delta +q_j\cos u_j^{(0)} \RC v_j=0 &\text{ in }\Omega,\\
	  v_j=g&\text{ on }\p \Omega,
	\end{cases}
\end{align}
for $j=1,2$. Now, by applying standard boundary determination  for the linear Schr\"odinger equation \eqref{first lin sine-Gordon}, one can determine that 
\[
q_1 \cos (f_0+\eps g)=q_1 \cos u_1^{(0)}=q_2\cos u_2^{(0)}=q_2 \cos (f_0+\eps g) \quad \text{ on }\p \Omega.
\]
In particular, for $\eps=0$, the above identity shows that 
\[
q_1\cos (f_0)=q_2\cos (f_0) \text{ on }\p \Omega.
\]
If $\cos (f_0)\equiv 0$ on $\p M$, we can slightly perturb $f_0$ so that there is $x_0\in \p \Omega$ with $\cos(f_0(x_0))\neq 0$ and repeat the above argument again. We deduce that $e^{\mathbf{i}\psi(x_0)}=1$, and since $\psi$ is constant, we conclude that
\[
 q_1=q_2 \text{ in } \Omega.
\]
%
This proves the claim.
\end{proof}

\appendix

\section{Proof of Proposition \ref{Prop: well-posed samll inputs}}\label{Appendix}

Let us prove Proposition \ref{Prop: well-posed samll inputs}. The proof is almost identical to the proof of Theorem \ref{Thm:wellposedness_and_expansion}, but Proposition \ref{Prop: well-posed samll inputs} does not exactly follow from Theorem \ref{Thm:wellposedness_and_expansion}. A very similar proof can be found from the work \cite[Section 2]{LLLS2019nonlinear}.

	\begin{proof}[Proof of Proposition \ref{Prop: well-posed samll inputs}]
	Let 
	\[
	\mathcal{B}_1=C^{2,\alpha}(\p \Omega), \quad \mathcal{B}_2=C^{\alpha}(\overline\Omega), \quad \mathcal{B}_3=C^{2,\alpha}(\overline{\Omega}), \quad \mathcal{B}_4=C^{\alpha}(\overline{\Omega})\times C^{2,\alpha}(\p \Omega)
	\]
	and consider the map 
	\begin{align*}
		\Psi:\mathcal{B}_1\times \mathcal{B}_2 \times \mathcal{B}_3 &\to \mathcal{B}_4, \\
		(f,F, u) &\mapsto \LC \Delta u +a(x,u)-F, u|_{\p \Omega}-f   \RC.
	\end{align*}
	Similar to \cite[Section 2]{LLLS2019nonlinear}, one can show that the map $u \mapsto a(x,u)$ is a $C^{\infty}$ map from $C^{2,\alpha}(\overline{\Omega}) \to C^{2,\alpha}(\overline{\Omega})$. 
	
	Notice that $ \Psi(0,0,0)  =(0,0)$, where we have the used condition \eqref{a(x,0)=0}.
	The first linearization of $\Psi=\Psi(f,F,u)$ at $(0,0,0)$ with respect to the variable $u$ is 
	\[
	\left. D_u \Psi \right|_{(0,0,0)}(v) = \LC \Delta v + \p_z a(x,0) v,v|_{\p \Omega}\RC,
	\]
	which is a homeomorphism $ \mathcal{B}_3\to  \mathcal{B}_4$ by the condition 
	\[
	0 \text{ is not a Dirichlet eigenvalue of }\Delta + \p _z a (x,0) \text{ in }\Omega.
	\]
	This is guaranteed by well-posedness and Schauder estimates for the linear second order elliptic equation .

	Now, the implicit function theorem in Banach spaces~\cite[Theorem 10.6 and Remark 10.5]{renardy2006introduction} yields that there are $\vareps,\delta>0$ and a neighborhood $\mathcal{N}_\delta\times \mathcal{A}_\vareps\subset C^{2,\alpha}(\p \Omega)\times C^{\alpha}(\overline\Omega)$ and a $C^{\infty}$ map $\mathcal{S}: \mathcal{N}_\delta\times \mathcal{A}_\vareps \to \mathcal{B}_3$ such that  
	\[
	\Psi(f,F,\mathcal{S}(f,F))=(0,0),
	\]
	whenever $\norm{f}_{C^{2,\alpha}(\p \Omega)}\leq \delta$ and $\norm{F}_{C^{\alpha}(\overline{\Omega})}\leq \vareps$.
	Since $\mathcal{S}$ is smooth and $\mathcal{S}(0,0) = 0$, the solution $u =\mathcal{S}(f,F)$ satisfies 
	\[
	\norm{u}_{C^{2,\alpha}(\overline{\Omega})}\leq C \LC \norm{f}_{C^{2,\alpha}(\p \Omega)}+\norm{F}_{C^{\alpha}(\overline\Omega)}\RC .
	\]
	Furthermore, by the uniqueness statement of the implicit function theorem, $u=\mathcal{S}(f,F)$ is the only solution to $\Psi(f,F,u)=(0,0)$ whenever $\norm{f}_{C^{2,\alpha}(\p \Omega)}+\norm{F}_{C^{\alpha}(\overline\Omega)}\leq \delta+\vareps$, and 
	$$
	\norm{u}_{C^{2,\alpha}(\overline{\Omega})}\leq  C\LC \varepsilon +\delta \RC.
	$$
	This can be achieved by redefining $\vareps,\delta>0$ to be smaller if necessary.
	As in \cite{LLLS2019nonlinear}, one can check that the solution operator $\mathcal{S}: \mathcal{N}_{\delta} \times \mathcal{A}_\vareps \to C^{2,\alpha}(\overline{\Omega})$ is a $C^{\infty}$ map in the Fr\'echet sense. The normal derivative is a linear map $C^{2,\alpha}(\overline{\Omega}) \to C^{1,\alpha}(\p \Omega)$. Thus for a fixed $F\in \mathcal{A}_\vareps$, $\Lambda_F: f\mapsto \p_\nu u_{f,F}$, where $u_{f,F}$ solves $\Delta u_{f,F}+a(x,u_{f,F})=F$ with $u_{f,F}|_{\p \Omega}=f$, is also a well defined $C^{\infty}$ map $\mathcal{N}_{\delta} \to C^{1,\alpha}(\p \Omega)$.
\end{proof}

\vskip0.5cm

\noindent\textbf{Acknowledgment.} 
T.L. was supported by the Academy of Finland (Centre of Excellence in Inverse Modeling and Imaging, grant numbers 284715 and 309963). Y.-H. Lin is partially  supported by the Ministry of Science and Technology Taiwan, under the Columbus Program: MOST-111-2628-M-A49-002.

\bibliographystyle{alpha}
\bibliography{ref}

\end{document}